\documentclass[reqno,12pt]{amsart}
\usepackage[margin=1.2in]{geometry}
\usepackage{amssymb}
\usepackage{amsthm}
\usepackage{amsmath}
\usepackage{amsxtra}
\usepackage{latexsym}
\usepackage{mathrsfs}
\usepackage[all,cmtip]{xy}
\usepackage[all]{xy}
\usepackage{enumitem}
\usepackage{xcolor}
\usepackage{comment}
\usepackage{mathabx,epsfig}
\usepackage{bm}
\usepackage{mathtools}
\usepackage{booktabs}
\usepackage{appendix}
\usepackage{hyperref}
\hypersetup{
    colorlinks,
    citecolor=black,
    filecolor=black,
    linkcolor=black,
    urlcolor=black
}

\usepackage[
backend=biber,
bibstyle=ieee-alphabetic,
citestyle=ieee-alphabetic,
sorting=nyt,
isbn=false,
url=false,
doi=false,
giveninits=true,
maxnames=10,
labelalpha=true,
maxalphanames=4,
dashed=false,
]{biblatex}
\DeclareCaseLangs{}

\DeclareFieldFormat*{date}{\mkbibparens{#1}}
\DeclareFieldFormat*{volume}{\mkbibbold{#1}}
\DeclareFieldFormat*{title}{\mkbibemph{#1\isdot}}
\DeclareFieldFormat*{journaltitle}{#1}
\DeclareFieldFormat*{pages}{{#1}}

\AtEveryBibitem{%
  \ifentrytype{online}{%
    \DeclareFieldFormat*{date}{#1}%
  }{%
  }%
}

\newtheorem{thm}{Theorem}[section]
\newtheorem{lem}[thm]{Lemma}

\newtheorem{claim}[thm]{Claim}
\newtheorem{prop}[thm]{Proposition}
\newtheorem{cor}[thm]{Corollary}

\theoremstyle{definition}
\newtheorem{conv}[thm]{Convention}

\newtheorem{eg}[thm]{Example}
\newtheorem{defn}[thm]{Definition}
\newtheorem{rmk}[thm]{Remark}
\newtheorem*{ack}{Acknowledgements}
\numberwithin{equation}{section}
\newcommand{\var}{\overline}

\theoremstyle{plain}
\newtheorem{theo}{Theorem}
\newtheorem{coro}[theo]{Corollary}

\newcommand{\pushoutcorner}[1][ul]{\save*!/#1-1.8pc/#1:(-1,1)@^{|-}\restore}


\def\Q{{\mathbb Q}}

\def\Z{{\mathbb Z}}
\def\P{{\mathbb P}}

\def\F{{\mathbb F}}
\def\m{{\mathfrak m}}

\def\var{\overline}
\def\Hom{\mathop{\mathrm{Hom}}\nolimits}

\def\sO{\mathcal{O}}

\DeclareMathOperator{\Spec}{Spec}
\DeclareMathOperator{\Proj}{Proj}

\DeclareMathOperator{\sht}{ht}

\newcommand{\cred}{\color{black}}

\AtBeginDocument{
  \def\MR#1{}
}

\newenvironment{claimproof}[0]
  {
   \paragraph{\it Proof.}
  }
  {
    \hfill$\blacksquare$
  }

\makeatletter
\def\@tocline#1#2#3#4#5#6#7{\relax
  \ifnum #1>\c@tocdepth 
  \else
    \par \addpenalty\@secpenalty\addvspace{#2}%
    \begingroup \hyphenpenalty\@M
    \@ifempty{#4}{%
      \@tempdima\csname r@tocindent\number#1\endcsname\relax
    }{%
      \@tempdima#4\relax
    }%
    \parindent\z@ \leftskip#3\relax \advance\leftskip\@tempdima\relax
    \rightskip\@pnumwidth plus4em \parfillskip-\@pnumwidth
    #5\leavevmode\hskip-\@tempdima
      \ifcase #1
       \or\or \hskip 1em \or \hskip 2em \else \hskip 3em \fi%
      #6\nobreak\relax
    \hfill\hbox to\@pnumwidth{\@tocpagenum{#7}}\par
    \nobreak
    \endgroup
  \fi}
\makeatother

\title[Fedder type criteria for quasi-$F$-splitting I]
{Fedder type criteria for quasi-$F$-splitting I}
\author{Tatsuro Kawakami, Teppei Takamatsu, and Shou Yoshikawa}
\address{Graduate School of Mathematical Sciences, University of Tokyo, 3-8-1 Komaba,
Meguro-ku, Tokyo 153-8914, Japan}
\email{kawakami@ms.u-tokyo.ac.jp}

\address{Department of Mathematics (Hakubi center), Graduate School of Science, Kyoto University, Kyoto 606-8502, Japan}
\email{teppeitakamatsu.math@gmail.com}

\address{Institute of Science Tokyo, 2-12-1 Ookayama, Meguro-ku, Tokyo 152-8550 Japan}
\email{yoshikawa.s.9fe9@m.isct.ac.jp}
\begin{document}

\begin{abstract}
Yobuko recently introduced the notion of quasi-$F$-splitting and quasi-$F$-split heights, which generalize and quantify the notion of Frobenius-splitting, and proved that quasi-$F$-split heights coincide with Artin-Mazur heights for Calabi-Yau varieties.
In this paper, we prove Fedder type criteria for quasi-$F$-splittings of complete intersections, and in particular, obtain a simple formula to compute Artin-Mazur heights of Calabi-Yau hypersurfaces.
As one of its applications, we prove that there exist Calabi-Yau varieties of arbitrarily high Artin-Mazur height over $\F_2$. 
We also give explicit defining equations of quartic K3 surfaces over $\mathbb{F}_{3}$ realizing all the possible Artin-Mazur heights.
\end{abstract}

\maketitle

\setcounter{tocdepth}{2}

\tableofcontents

\section{Introduction}
In both commutative algebra and algebraic geometry in positive characteristic, the study of Frobenius maps {\cred has} led to significant developments.
The theory of \textit{Frobenius splitting} ($F$-splitting for short), introduced by Mehta-Ramanathan \cite{mehta--ramanathan}, is one of such developments.
We say that a scheme $X$ of positive characteristic is \textit{$F$-split} if the Frobenius map $F\colon \sO_X\to F_{*}\sO_X$ splits as an $\sO_X$-module homomorphism.
$F$-split varieties have been studied by many authors because they satisfy very good properties.
For example, the Kodaira vanishing theorem, which can fail in positive characteristic (\cite{Ray78}), holds on smooth projective $F$-split varieties. 
In addition, $F$-split varieties are known to lift to the ring of Witt vectors of length two (\cite{Joshi2007}, Bhatt's proof in \cite{Lan15}).

The $F$-splitting property is also famous for its interesting connections with other invariants. 
From now on, unless otherwise stated, varieties are defined over an algebraically closed field of positive characteristic.
First, an abelian variety is $F$-split if and only if it is ordinary.
Next, we focus on Calabi-Yau varieties, which are varieties with trivial canonical sheaf and trivial middle cohomologies of structure sheaves.
We can define the \textit{Artin-Mazur height} as the height of the formal group arising from the variety in the sense of Artin-Mazur, which takes a value in $\Z_{>0}\cup \{\infty\}$ (see \cite{AM}, \cite{GK2}). 
It is worth noting that the Artin-Mazur height has a significant role in a stratification of the moduli space of K3 surfaces (\cite[(15.1) Theorem]{GK}), and the height is still a conspicuous invariant for higher-dimensional Calabi-Yau varieties.
For this reason, many authors studied the structure of the formal group and computed the Artin-Mazur heights, especially for K3 surfaces (see \cite{Stienstra}, \cite{Goto}, \cite{Yui}, and \cite{Kedlaya}).
Using a characterization of the Artin-Mazur height by van der Geer-Katsura \cite[Theorem 2.1]{GK2}, we can easily see that the Artin-Mazur height is equal to one if and only if $X$ is $F$-split.

Now, let us recall a criterion for $F$-splitting, famous as \textit{Fedder's criterion} (\cite{Fedder}).
This criterion asserts that, for a hypersurface $\{f=0\} \subseteq \P^n_{[x_0:\cdots:x_n]}$, the $F$-splitting of the hypersurface is equivalent to $f^{p-1} \notin (x^p_0,\ldots,x^p_n)$.
Using this criterion, we can easily confirm whether a given Calabi-Yau hypersurface is of Artin-Mazur height one.

On the other hand, Calabi-Yau varieties behave pathologically, especially when the Artin-Mazur heights are infinite, rather than when the Artin-Mazur heights are greater than one.
Based on this fact, Yobuko introduced the notion of a \textit{quasi-$F$-splitting}, a generalization of that of $F$-splitting, which distinguishes whether the Artin-Mazur heights of Calabi-Yau varieties are infinite or not. 
He defined the \emph{quasi-$F$-split height $\sht(X)$} of a scheme $X$ of positive characteristic by the infimum number $n>0$ such that there exists a $W_n\sO_X$-module homomorphism $\phi$ which makes the following diagram commutative
\[
\xymatrix{
W_n\sO_X \ar[r]^-{F} \ar[d]_-{R^{n-1}} & F_* W_n \ar@{-->}[ld]^{\exists \phi} \sO_X \\
\sO_X,
}
\]
where $R$ is the restriction map and $F$ is the Frobenius map (see {\cred Subsection \ref{subsec:Witt vectors} for details}). 
Furthermore, we say that $X$ is \emph{quasi-$F$-split} if and only if the quasi-$F$-split height $\sht(X)$ is finite.
We note that $X$ is $F$-split if and only if $\sht(X)=1$, and therefore quasi-$F$-splitting is a generalization of $F$-splitting.
Yobuko \cite[Theorem 4.5]{Yobuko19} proved that the quasi-$F$-split height coincides with the Artin-Mazur height for Calabi-Yau varieties. 
In particular, the quasi-$F$-splitting property distinguishes whether the Artin-Mazur height of Calabi-Yau varieties are infinite as desired.

Here, it is worth mentioning that  
quasi-$F$-splitting is a good generalization even if it is not limited to Calabi-Yau varieties.
Indeed, quasi-$F$-split varieties have remarkable properties similar to those of $F$-split varieties; the Kodaira vanishing theorem holds on them (\cite{Yobuko2}, \cite{nakkajima2021degenerations}), and 
every quasi-$F$-split variety lifts to $W_2(k)$ (\cite{Yobuko19}, \cite{AZ21}).

Given these backgrounds, it is important to find out which varieties are quasi-$F$-splits, and therefore, a criterion for quasi-$F$-splitting is highly desirable.
The main aim of this paper is to generalize Fedder's criterion to quasi-$F$-splitting.

\subsection{Fedder type criteria for quasi-$F$-splitting}
\label{subsection: Fedder type criteria}
It is difficult to confirm quasi-$F$-splitting by its definition, even in the case of hypersurfaces.
One of the difficulties comes from the complexity of the ring structure of the ring of Witt vectors.
As $n$ increases, the computation of the ring of Witt vectors of length $n$ becomes more difficult.

Surprisingly, our Fedder type criterion for quasi-$F$-splitting needs only the ring structure of the ring of Witt vectors of length two.
Therefore, the criterion gives us a very simple way to determine the quasi-$F$-split heights of hypersurfaces, and in particular, the Artin-Mazur heights of Calabi-Yau hypersurfaces.

We start with the preparation of the notation for the criterion.
Let $k$ be a perfect field $k$ of characteristic $p>0$, $S \coloneq k[x_1,\ldots,x_N]$ a polynomial ring, $\m:=(x_1,\ldots,x_N)$, and $R:=S_\m$.
We take a basis
\[
\{F_* {\cred (} x_1^{i_1}\cdots x_N^{i_N} {\cred )} \mid 0 \leq i_1,\ldots,i_N \leq p-1 \}
\]
of $F_*S$ over $S$ and denote the dual of $F_*{\cred (}(x_1 \cdots x_N)^{p-1} {\cred )}$ with respect to this basis by $u$.
Then $u$ is a generator of $\Hom_S(F_*S,S)$.
We define the map
\[
\Delta_1 \colon S \to S
\]
as follows:
Let $a \in S$. We take the monomial decomposition $a=\sum^m_{i=1} M_i$.
We define $\Delta_1(a) \in S$ by 
\[
(0,\Delta_1(a))=(a,0)-\sum (M_i,0)\ \hspace{0.2cm} \text{in} \hspace{0.1cm} W_2(S).
\]
We can compute that 
\[
\Delta_1(a)=\sum_{\substack{0 \leq \alpha_1, \ldots,\alpha_m \leq p-1 \\ \alpha_1+\cdots+\alpha_m=p}} \frac{1}{p} \binom{p}{\alpha_1, \ldots ,\alpha_m}(M_1)^{\alpha_1} \cdots (M_m)^{\alpha_m}.
\]

Now, we can introduce a Fedder type criterion for quasi-$F$-splitting, one of our main theorems.

\begin{theo}[Fedder type criterion for quasi-$F$-splitting (see Theorem \ref{thm:Fedder's criterion} and Corollary \ref{cor:Fedder's criterion for quasi-F-splitting} for a more general statement)]\label{Intro:thm:Fedder's criterion}
Let $f\in S$ and $\theta$ an $S$-module homomorphism defined by 
\[
\theta \colon F_*S \to S \ ;\ F_*a \mapsto u(F_*{\cred (} \Delta_1(f^{p-1})a {\cred )}).
\]
We define an increasing sequence $\{I_n\}_n$ of ideals by $I_1:=(f^{p-1})$ and 
\[
I_{n+1}:=\theta(F_*I_n \cap \mathrm{ker}(u))+I_1,
\]
inductively.
Then we have
\[
\sht(R/f)=\mathrm{inf}\{n \mid I_n \nsubseteq \m^{[p]} \},
\]
{\cred where $\inf \emptyset := \infty$.}
Furthermore, if $f$ is a homogeneous element and $N \geq 3$,
then we have
\[
\sht(\Proj(S/f))=\mathrm{inf}\{n \mid I_n \nsubseteq \m^{[p]} \}.
\]
\end{theo}
\noindent
Theorem \ref{Intro:thm:Fedder's criterion} holds for more general settings.
We refer to Theorem \ref{thm:Fedder's criterion} for the case of complete intersection local rings, and Theorem \ref{thm:Fedder's criterion for projective varieties in weighted case} for the case of complete intersections in a weighted projective space. 

Since $\{I_n\}_n$ is an increasing sequence, we find the maximum element $I_\infty$ of ideals $\{I_n \mid n\}$ by the Noetherian property.
Using $I_\infty$, we obtain another criterion, which is useful to confirm non-quasi-$F$-splitting of a given hypersurface.

\begin{coro}[cf.~Corollary \ref{cor:Fedder's criterion for quasi-F-splitting}]\label{Intro:cor:criterion for quasi-F-splitting}
We take $f \in S$, $\theta$, and $\{I_n\}_n$ as in Theorem \ref{Intro:thm:Fedder's criterion}.
Let $I_\infty$ be the maximum element of $\{I_n \mid n\}$.
Then the following hold.
\begin{itemize}
    \item[\textup{(1)}] $R/f$ is quasi-$F$-split if and only if $I_\infty$ is not contained in $\m^{[p]}$.
    \item[\textup{(2)}] $I_\infty$ is the minimum element of the set of of ideals $J$ satisfying
    \[
    J \supseteq \theta(F_*J \cap \mathrm{ker}(u))+(f^{p-1}).
    \]
\end{itemize}
\end{coro}

To find the quasi-$F$-split height, Theorem \ref{Intro:thm:Fedder's criterion} gives us a simpler way than confirming the existence of $\phi$ in the definition of the quasi-$F$-split height.
However, {\cred compared} with the original Fedder's criterion, {\cred it remains complicated for the following reasons:}
We have to deal with a sequence of ideals $\{I_n\}_n$ in Theorem \ref{Intro:thm:Fedder's criterion}, whereas
it is sufficient to treat only one element $f^{p-1}$ in the original Fedder's criterion.

Fortunately, in the case where the canonical divisor is trivial, we obtain a simpler criterion that is sufficient to deal with only one element.

\begin{theo}[Calabi-Yau case of a Fedder type criterion for quasi-$F$-splitting (see Theorem \ref{thm:Fedder's criterion for Calabi-Yau} for a more general statement)]\label{Intro:thm:Fedder's criterion for Calabi-Yau}
Let $f\in S:=k[x_1,\ldots,x_N]$ be a homogeneous of degree $N$ and $\theta$ an $S$-module homomorphism defined by
\[
\theta \colon F_*S \to S \ ;\ F_*a \mapsto u(F_*{\cred (} \Delta_1(f^{p-1})a {\cred )}).
\]
Then we have
\[
\sht(S/f)=\mathrm{inf}\{n \mid \theta^{n-1}(F^{n-1}_*f^{p-1}) \notin \m^{[p]} \}.
\]
Furthermore, if $N \geq 3$, then we have
\[
\sht(\Proj(S/f))=\mathrm{inf}\{n \mid \theta^{n-1}(F^{n-1}_*f^{p-1}) \notin \m^{[p]} \}.
\]
\end{theo}
\noindent
We note that $\Proj(S/f)$ is a Calabi-Yau hypersurface if it is smooth since $f\in S$ is homogeneous of degree $N$, and therefore Theorem \ref{Intro:thm:Fedder's criterion for Calabi-Yau} enables us to compute the Artin-Mazur heights of Calabi-Yau hypersurfaces.

We close this subsection showing applications of Theorem \ref{Intro:thm:Fedder's criterion}, Corollary \ref{Intro:cor:criterion for quasi-F-splitting}, and Theorem \ref{Intro:thm:Fedder's criterion for Calabi-Yau} as follows:

\begin{description}
    \item[Calabi-Yau hypersurface of arbitrary large Artin-Mazur height] 
     Using Theorem \ref{Intro:thm:Fedder's criterion for Calabi-Yau}, we found an explicit defining equation of Calabi-Yau hypersurfaces over $\mathbb{F}_2$ whose Artin-Mazur height $2h$ for every positive integer $h \in \Z_{>0}$ (Example \ref{eg:unboundedness}).
     Furthermore, we apply Theorem \ref{Intro:thm:Fedder's criterion for Calabi-Yau} to construct a Calabi-Yau hypersurface over $\var{\F}_p$ whose Artin-Mazur height $h$ for every positive integer $h\in\Z_{>0}$ and a prime number $p$ in \cite{KTY2}. 
    \item[Use of a computer algebra system]
Theorem \ref{Intro:thm:Fedder's criterion} and Theorem \ref{Intro:thm:Fedder's criterion for Calabi-Yau} are well-suited to a computer algebra system
(see \cite{Takamatsu_code} for the Macaulay 2 code).
    For example, we found an explicit equation of a quartic K3 surface in $\mathbb{P}^3_{\mathbb{F}_{3}}$ of the Artin-Mazur height $h$ for every $h\in \{1,\ldots, 10, \infty \}$ (Example \ref{example:char 3 K3 surface}). We remark that such examples have been known only over $\mathbb{F}_{2}$ (see \cite{Kedlaya}). 
    We also found an explicit equation of a Calabi-Yau quintic threefold over $\mathbb{F}_2$ of Artin-Mazur height $60$ (Example \ref{eg:ht=60}).

\end{description}

\subsection{Strategy of the proof of Theorem \ref{Intro:thm:Fedder's criterion}}
\label{subsect:strategyofpf}

Theorem \ref{Intro:thm:Fedder's criterion} is proved in Sections \ref{Section:Construction of splitting maps} and \ref{Section:Fedder type criteria for quasi-$F$-splitting}.
{\cred
In this subsection, we primarily explain Section \ref{Section:Construction of splitting maps}, which contains arguments that fundamentally differ from the proof of Fedder’s original criterion.
}

In the proof of the original Fedder's criterion, it is important to describe the explicit structure of $\Hom_S(F_*S,S)$ as
 \begin{align}\label{introduction:isom}
    F_*S \to \Hom_S(F_*S,S)\ ;\ F_{\ast}g  \mapsto \psi_g:=(F_*a \mapsto u(F_*{\cred (}ga {\cred )})),
\end{align}
where $u$ is the dual of $F_*(x_1^{p-1} \cdots x_N^{p-1})$.
For this reason, the aim of Section \ref{Section:Construction of splitting maps} is to clarify a precise description of $\Hom_{W_n(S)}(F_*W_n(S),S)$.

\textbf{Step 1.} 
In this step, we introduce an $S$-module $Q_{S,n}$ to describe $\Hom_{W_n(S)}(F_*W_n(S),S)$. 

Since the ring structure of $W_n(S)$ is complicated, it is hard to confirm whether a given map is a $W_n(S)$-module homomorphism. 
For this reason, we take the pushout $Q_{S,n}$ of the diagram
\[
\xymatrix{
W_n(S) \ar[r]^-{F} \ar[d]_{R^{n-1}} & F_*W_n(S) \ar[d] \\
S \ar[r] & Q_{S,n} \pushoutcorner
}
\]
as $W_n (S)$-modules.
Then $Q_{S,n}$ has the natural $S$-module structure, and the lower horizontal map can be seen as an $S$-module homomorphism (see Definition \ref{defn:Q}).
By the definition of the pushout, we can see that $X$ is $n$-quasi-$F$-split  (see Definition \ref{defn:quasi-F-split} for the definition) if and only if the bottom horizontal map splits.
Therefore, the problem is reduced to studying the structure of $\Hom_S(Q_{S,n},S)$.

\textbf{Step 2.}
In this step, we construct an {\cred $S$-module} isomorphism 
\begin{equation}\label{isom:Psi1}
F_*S \oplus \bigoplus_{2 \leq e \leq n} F^{e-1}_*({\cred \mathrm{ker}(u)}) \xrightarrow{\cong} \Hom_S(Q_{S,n},S).
\end{equation}

By Proposition \ref{prop:exact sequence of Q_n's}, we have an exact sequence {\cred of $W_n(S)$-modules}
\begin{equation}\label{exact sequence:Q}
\xymatrix{
0 \ar[r] & F^{n-1}_*(F_* S /{\cred S}) \ar[r]^-{V^{n-1}} & Q_{S,n} \ar[r]^-{\pi} & Q_{S,n-1} \ar[r] & 0
}
\end{equation}
for $n\geq 2$,
where $V$ is the map induced by the Verschiebung morphism, $\pi$ is the map induced by the restriction map $R$ (see Definition \ref{defn:Witt ring} and Definition \ref{defn:Q}), {\cred and $F_*S/S$ is the cokernel of the Frobenius homomorphism $F \colon S \rightarrow F_*S$.}

To prove (\ref{isom:Psi1}), we {\cred construct} a splitting map $\sigma_n$ of $V^{n-1}$ in (\ref{exact sequence:Q}).
For an element $a \in S$, we take the monomial decomposition $a=\sum M_i$, and define $\delta_i(a) \in S$ by
\[
(0,\delta_1(a),\delta_2(a),\ldots)=[a]-\sum [M_i]\hspace{0.2cm} \text{in} \hspace{0.15cm} W(S).
\]
We note that $\delta_1(a)$ coincides with $\Delta_1(a)$ in Theorem \ref{Intro:thm:Fedder's criterion}.
We define $\sigma_2$ by
\[
\sigma_2 \colon Q_{S,2} \to F_*(F_*S/{\cred F(S)})\ ;\ {\cred F_*(a,b) \mapsto F^2_* (\delta_1(a)+b).}
\]
Then we have $\sigma_2 \circ V=\mathrm{id}$.
On the other hand, it is not clear {\cred whether} $\sigma_2$ is an $S$-module homomorphism. We refer to Proposition \ref{prop:delta gives a splitting} for the proof.
We define $\sigma_3$ by
\[
{\cred
\xymatrix@R=3pt{
\sigma_3 \colon
Q_{S,3} \ar[r]  & F_* Q_{S,2} \ar[r]^{\sigma_2} & F^2_*( F_*S/F(S)), \\
F_*(a,b,c) \ar@{|->}[r] \ar@{(-}[u] & F_*^2((\delta_1(a),\delta_2(a))+(b,c)) \ar@{|->}[r] \ar@{(-}[u] & F^3_*((\delta^2_1(a)+\delta_2(a))+\delta_1(b)+c) \ar@{(-}[u] 
}
}
\]
{\cred and can define $\sigma_n$ similarly; we then obtain the desired splitting.}

By $\sigma_n$ and (\ref{exact sequence:Q}), we have
\begin{align}\label{introduction:splitting isom}
\Hom_S(Q_{S,n},S) \xrightarrow{(\sigma_n^*,\pi^*)}& \Hom_S(F^{n-1}_*(F_*S/{\cred F(S)}),S) \oplus \Hom_S(Q_{S,n-1},S) \notag\\
 \cong& F^{n-1}_*({\cred \mathrm{ker}(u)})  \oplus \Hom_S(Q_{S,n-1},S),
\end{align}
where the second isomorphism follows from (\ref{introduction:isom}).
Repeating this procedure, we obtain an {\cred $S$-module} isomorphism
\[
\Psi_n \colon F_*S \oplus \bigoplus_{2 \leq e \leq n} F^{e-1}_*({\cred \mathrm{ker}(u)}) \xrightarrow{\cong} \Hom_S(Q_{S,n},S),
\]
which is the goal of this step.

\textbf{Step 3.} In this step, we define  maps $\Delta_s\colon S\to S$ to study 
\[
\psi_{(g_1,\ldots,g_n)}:=\Psi_n({\cred (} {\cred F_*}g_1, {\cred F^2_* g_2,}\ldots, {\cred F^n_*}g_n {\cred )})\in \Hom_S(Q_{S,n},S)
\] 
for $({\cred F_*}g_1, {\cred F^2_* g_2,}\ldots, {\cred F^n_*}g_n) \in F_*S \oplus \bigoplus_{2 \leq e \leq n} F^{e-1}_*({\cred \mathrm{ker}(u)})$ and $s \in \Z_{>0}$.

We define maps $\Delta_s$ by $\Delta_1:=\delta_1$ and
\begin{equation}
    \label{eqn:Deltadelta}
\Delta_s(a):=\sum_{r=1}^{s-1} \Delta_r \circ \delta_{s-r}(a)+\delta_s(a)
\end{equation}
inductively.
Then a direct computation shows
\begin{align}\label{equation:sigma=Delta}
     \sigma_n((a_0,a_1,\ldots,a_{n-1}))= \Delta_{n-1}(a_0)+\Delta_{n-2}(a_1)+\cdots+a_{n-1}.
\end{align}
Combining the construction of $\Psi_n$ with (\ref{equation:sigma=Delta}), we obtain 
\begin{align}\label{equation:psi}
  \psi_{(g_1,\ldots,g_n)}(V^{r-1}[a]))=u^r(F^r_*{\cred (}g_ra {\cred )})
  +u^{r+1}(F^{r+1}_* {\cred (}g_{r+1}\Delta_1(a) {\cred )})+\cdots+u^n(F^n_*{\cred (}g_n\Delta_{n-r}(a){\cred )})  
\end{align}
for all $1 \leq r \leq n$, where $[a]:=(a,0,\ldots,0)$ denotes the Teichm\"uller lift of $a$.

\textbf{Step 4.}
As the last step, we find simple descriptions of $\Delta_s$.

For the computation of the right-hand side of (\ref{equation:psi}), it suffices to know $\Delta_{s}(a)$ modulo $F(S)$.
In Theorem \ref{Intro:thm:delta formula}, we will see that $\Delta_s(a)$ can be written only by $\Delta_{1}(a)$ modulo $F(S)$. 
This fact implies that only the ring structure of $W_2(S)$ is necessary for the computation of quasi-$F$-split heights.

\begin{theo}[Delta formula ({\cred Remark \ref{eg:explicit delta} and }Theorem \ref{thm:delta formula})]\label{Intro:thm:delta formula}
$\Delta_s(a) \equiv a^{p^s-p}\Delta_1(a) \mod F(S)$ for all $a \in S$ and $s \geq 1$.
\end{theo}

\noindent
Now, we see the sketch of proof of Theorem \ref{Intro:thm:delta formula}.
We recall that the ring of Witt vectors is defined via mixed characteristic (see Subsection \ref{subsec:Witt vectors}). 
For this reason, we consider the ring of Witt vectors over a lift $A:=W(k)[x_1,\ldots,x_N]$ of $S$.
Taking the monomial decomposition, we can define $\delta_n$ and $\Delta_n$ on $A$ in a way similar to the above.
Then Theorem \ref{Intro:thm:delta formula} can be reduced to proving an equation
\begin{align}\label{introduction:eq:delta}
    \Delta_s(a)=\frac{a^{p^s}-(a^p-p\Delta_1(a))^{p^{s-1}}}{p^s}
\end{align}
for $a \in A$.
Indeed, the right-hand side is equal to 
\[
a^{p^s-p}\Delta_1(a)-\frac{1}{p^{s-2}}\binom{p^{s-1}}{2}a^{p^s-2p}\Delta_1(a)+\cdots +(-1)^{p+1}p^{p^{s-1}-s}\Delta_1(a)^p,
\]
{\cred and its image under the map $A \to S/F(S)$ is $a^{p^s-p}\Delta_1(a)$. By making essential use of the ghost component—which behaves well only in mixed characteristic (see Lemma \ref{lem:delta formula in mixed characteristic} for details)—we prove \eqref{introduction:eq:delta} and complete the proof of Theorem \ref{Intro:thm:delta formula}.}

\subsection{Other applications of the main theorems}

In a subsequent paper \cite{KTY2}, we apply Theorems \ref{Intro:thm:Fedder's criterion} and \ref{Intro:thm:Fedder's criterion for Calabi-Yau} as follows:

\begin{description}
\item[Rational double points]
    Rational double points (RDPs, for short) in characteristic $p$ are $F$-split when $p>5$, while {\cred they} can be non-F-split when $p\leq 5$ (\cite{Hara2}).
    Applying Theorem \ref{Intro:thm:Fedder's criterion}, we determine quasi-$F$-split heights of all RDPs, and in particular, show that RDPs are all quasi-$F$-split. 
    We remark that the quasi-$F$-split heights of RDPs {\cred were already computed by Yobuko in his unpublished work, except for those of} $D_{n}$-type in $p=2$.
    His proof depends on detailed calculations of local cohomologies, and the method {\cred is completely different from ours}.

\item[Fano varieties]
    Using Theorem \ref{Intro:thm:Fedder's criterion}, we can find a non-quasi-$F$-split smooth Fano $d$-fold for all $d>2$. 
    {\cred Note that Totaro (\cite{Totaro}) already constructed non-quasi-$F$-split smooth Fano varieties, but our example gives a different one.}
    On the other hand, we prove smooth del Pezzo surfaces (i.e., smooth Fano varieties of dimension two) have quasi-$F$-split heights at most two, and in particular, they are all quasi-$F$-split. Note that there exist non-$F$-split del Pezzo surfaces when $p\leq 5$ (\cite[Example 5.5]{Hara}).
    
\item[Inversion of adjunction]
We see that the inversion of adjunction for quasi-$F$-spitting does not hold in general; there is a non-quasi-$F$-split hypersurface in a quasi-$F$-split variety.
 On the other hand, we prove that $X \subseteq \P^N$ is a hypersurface of degree $d$ and $Y$ an intersection of general $(N+1-d)$ hyperplane sections, then $\sht(Y) \geq \sht(X)$. Here, we use Theorem \ref{Intro:thm:Fedder's criterion for Calabi-Yau} essentially.
 This result also holds for hypersurfaces in weighted projective spaces, and it is useful to {\cred give an} upper bound of quasi-$F$-split heights of some class of Fano varieties. 
 \item[Fiber product]
We see that if $X$ and $Y$ are complete intersections and not $F$-split, then $X \times Y$ is not quasi-$F$-split.
 On the other hand, we prove {\cred that} if $X$ is $F$-split, then $\sht(X \times Y)=\sht(Y)$, where $X$ and $Y$ are not necessarily complete intersections.
 For the proof, we use a generalized version of Theorem \ref{Intro:thm:Fedder's criterion}.
 \item[General fiber] 
{\cred We observe that, in general, the quasi-$F$-splitting property does not descend to a general fiber}; we show that there exists a smooth Fano threefold of $\sht(X)=2$ with a wild conic bundle structure.
 On the other hand, using Theorem \ref{Intro:thm:Fedder's criterion for Calabi-Yau}, we can see that the quasi-$F$-splitting property is inherited to a general fiber when the generic fiber has trivial canonical divisor and is a complete intersection in a projective space over the function field of the base scheme.
 In particular, we show that every quasi-$F$-split surface has no quasi-elliptic fibrations.
\end{description}

\subsection{Related work}
\begin{itemize}
\item Hiromu Tanaka, Jakub Witaszek, Fuetaro Yobuko, and the authors investigated birational geometric aspects of quasi-$F$-splitting in \cite{KTTWYY}, \cite{KTTWYY2}. In particular, we proved that two-dimensional klt singularities over a perfect field of any characteristic and three-dimensional $\Q$-factorial klt singularities over a perfect field of characteristic $p>41$ are quasi-$F$-split.
\item 
{\cred
Tanaka-Witaszek-Yobuko introduced the notion of quasi-$F^e$-splitting (the iterated version of quasi-$F$-splitting), quasi-$F$-regularity, and quasi-$+$-regularity in the work \cite{TWY}. They also prove the inversion of adjunction in general settings.
}

\item 
{\cred The third author proves a Fedder type criterion for quasi-$F^e$-splitting and quasi-$F$-regularity in \cite{YoshikawaFedder}.}
\item
{\cred
Yobuko (\cite{YobukoQuasiHodgeWitt}) proved a generalization of the result on fiber products in \cite{KTY2} without using Fedder’s criterion.
}
\end{itemize}

\subsection{Notation and terminologies}
{\cred Throughout the paper, we fix a prime number $p$.}
Let $S$ be a ring of characteristic $p$.
\begin{itemize}
    \item $F \colon S \to S$ is the absolute Frobenius homomorphism.
    If the Frobenius morphism is finite, then we say that $S$ is \emph{$F$-finite}.
    \item For an $S$-module $M$, we define another $S$-module structure on $M$ by
    \[
    a \cdot m:=a^pm
    \]
    for $a \in S, m \in M$,
    which is denoted by $F_*M$.
    In order to distinguish the action by $S$ on $F_*M$ from that on $M$, we denote the elements of $F_*M$ as $F_*m$.
    Thus 
    \[
    aF_*m=F_* {\cred (}a^pm {\cred )}
    \]
    for $a \in S, m \in M$.
    We can view the Frobenius map
    \[
    S \to F_*S\ ;\ a \mapsto F_*a^p=aF_*1
    \]
    as an $S$-module homomorphism.
    We similarly define the iterated version $F^n_*M$ and $F^n_*m$ as above.
    \item
    {\cred We denote the cokernel of the Frobenius homomorphism
    \[
    F \colon S \rightarrow F_*S
    \]
    by $F_*S/S$ by an abuse of notation.
    We apply the same notational abbreviation to the structure sheaf $\sO_X$ of a scheme $X$ of characteristic $p$.
    }
    \item Let $I$ be an ideal of $S$.
    We define $I^{[p]}$ as {\cred the} ideal generated by
    $\{a^p \mid a \in I\}$.
    Furthermore, we inductively define $I^{[p^n]}:=(I^{[p^{n-1}]})^{[p]}$ for $n\in\Z_{>0}$.
    Note that $I \cdot F^n_*S=F^n_*I^{[p^n]}$.
    \item We say a scheme $X$ is a {\em variety} if $X$ is an integral scheme 
    that is separated and of finite type over a field $k$. 
    \item Let $K$ be a{\cred n} {\cred $F$-finite} field of characteristic $p > 0$.
    A subset $\{x_{i}\}_{i} \subseteq K$ is called a {\em p-basis of $K$} if 
    \[\{F_*(\prod_{i} x_{i}^{e_{i}})|0 \leq e_{i} \leq p-1
    \}\]
    form a basis of $F_*K$ over $K$.
    \item Let $A$ be a ring.
    Since the polynomial ring $A[x_1,\ldots,x_N]$ is a free $A$-module and the monomials form a basis,
    we obtain the unique decomposition
    \[
    f=\sum a_i M_i,
    \]
    for $f \in A$, 
    where $a_i \in A$ and $M_i$ is a monic monomial such that $M_i \neq M_j$ if $i \neq j$.
    The decomposition is called the \emph{monomial decomposition} of $f$ over $A$.
\end{itemize}

\begin{ack}
The authors wish to express their gratitude to Hiromu Tanaka, Fuetaro Yobuko, and Jakub Witaszek for valuable discussion.
They are also grateful to Shunsuke Takagi, Kenta Sato, Masaru Nagaoka, Naoki Imai, Tetsushi Ito, Yuya Matsumoto, Ippei Nagamachi, and Yukiyoshi Nakkajima and Jack J. Garzella for helpful comments.
{\cred
Finally, the authors would like to express their sincere gratitude to the referee for carefully reviewing the manuscript and providing valuable comments.}
Kawakami was supported by JSPS KAKENHI Grant number JP19J21085, JP22KJ1771, and JP24K16897.
Takamatsu was supported by JSPS KAKENHI Grant number JP19J22795. 
Yoshikawa was supported by JSPS KAKENHI Grant number JP20J11886.
\end{ack}

\section{Preliminaries}

\subsection{The ring of Witt vectors}\label{subsec:Witt vectors}
First, we recall the definition and generalities of the ring of Witt vectors. The basic reference is \cite[II, \S 6]{Serre}.

\begin{lem}
\label{lem:witt polynomial}
For $n\in \Z_{\geq 0}$, we define the polynomial ${\cred w}_{n} \in \Z[X_{0},\ldots X_{n}]$ by
\[
{\cred w}_{n} := \sum_{i=0}^{i=n} p^i X_{i}^{p^{n-i}}.
\]
Then, there exist polynomials $S_{n} ,P_{n} \in \Z[X_{0},\ldots, X_{n}, Y_{0}, \ldots, Y_{n}]$ for any $n \in \Z_{\geq 0}$ such that
\[
{\cred w}_{m} (S_{0}, \ldots, S_{m}) = {\cred w}_{m}(X_{0}, \ldots X_{m}) +{\cred w}_{m}(Y_{0}, \ldots Y_{m})
\]
and 
\[
{\cred w}_{m} (P_{0}, \ldots, P_{m}) = {\cred w}_{m}(X_{0}, \ldots X_{m}) \cdot {\cred w}_{m}(Y_{0}, \ldots Y_{m})
\]
hold for any $m \in \Z_{\geq 0}$.
\end{lem}

\begin{proof}
See \cite[II, Theorem 6]{Serre}.
\end{proof}

\begin{defn}
\label{defn:Witt ring}
Let $A$ be a ring.
We put
\[
W(A) := \{
(a_{0}, \ldots, a_{n}, \ldots) | a_{n} \in A
\} = \prod_{\Z_{\geq 0}} A,
\]
equipped with the addition
\[
(a_{0}, \ldots, a_{n}, \ldots) + (b_{0}, \ldots, b_{n}, \ldots) := (S_{0}(a_{0},b_{0}), \ldots, S_{n}(a_{0}, \ldots, a_{n}, b_{0}, \ldots, b_{n}), \ldots)
\]
and the multiplication
\[
(a_{0}, \ldots, a_{n}, \ldots) \cdot (b_{0}, \ldots, b_{n}, \ldots) := (P_{0}(a_{0},b_{0}), \ldots, P_{n}(a_{0}, \ldots, a_{n}, b_{0}, \ldots, b_{n}), \ldots).
\]
Then $W(A)$ is a ring with $1= (1, 0, \ldots) \in W(A)$, which is called the ring of Witt vectors over $A$.
We define an additive map $V$ by
\[
V \colon W(A) \rightarrow W(A); (a_{0}, a_{1}, \ldots) \mapsto (0,a_{0}, a_{1}, \ldots).
\]

We define the ring $W_{n}(A)$ of Witt vectors over $A$ of length $n$ by
\[
W_{n}(A) := W(A)/ V^{n}W(A).
\]
We note that $V^{n}W(A)\subseteq W(A)$ is an ideal. 
We define a $W(A)$-module homomorphism $R$ by
\[
W(A) \rightarrow W_{n}(A) ; (a_{0}, a_{1}, \ldots) \mapsto (a_{0}, a_{1}, \ldots, a_{n-1})
\]
and call $R$ the restriction map.
Moreover, we define 
\begin{eqnarray*}
V \colon W_{n}(A) \rightarrow W_{n+1}(A),
\end{eqnarray*}
and
\[
R \colon W_{n+1} (A) \rightarrow W_{n}(A),
\]
in a similar way.
The image of $\alpha \in W(A)$ or $W_n(A)$ by $V$ is denoted by $V\alpha$.

For any ring homomorphism $f \colon A \to A'$, a ring homomorphism
\[
W(A) \to W(A')\ ;\ (a_0,a_1,\ldots) \mapsto (f(a_0),f(a_1),\ldots)
\]
is induced, and the induced homomorphism is commutative with $R$ and $V$.
Note that $f$ also induces a ring homomorphism $W_n(A) \to W_n(A')$.

For an ideal $I$ of $A$,
we denote the kernels of $W(A) \to W(A/I)$ and $W_n(A) \to W_n(A/I)$ induced by the natural quotient map $A \to A/I$ are denoted by $W(I)$ and $W_n(I)$, respectively.
Suppose that $A$ is a ring of characteristic $p>0$.
{\cred The ring homomorphism} 
\[
W(A) \to W(A),\ \text{and}\ W_n(A) \to W_n(A)
\]
{\cred induced}
by the Frobenius morphism is also denoted by $F$ and called the Frobenius morphism.
The image of $\alpha \in W(A)$ or $W_n(A)$ by $F$ is denoted by $F\alpha$.

We have
\[
FV = VF = p \in \Hom (W(A), W(A)).
\]
We define rings $\var{W}(A)$ and $\var{W}_n(A)$ by
\[
\var{W}(A):=W(A)/pW(A),\ \text{and}\ \var{W}_n(A):=W_n(A)/pW_n(A).
\]
Then we can define maps $V$, $R$, and $F$ on $\var{W}(A)$ and $\var{W}_n(A)$ in a way similar to the case of $W(A)$ and $W_n(A)$.
We can also define ideals $\var{W}(I)$ and $\var{W}_n(I)$ in a similar way.

Let $X$ be an $\F_p$-scheme.
We define sheaves $W_n\sO_X$ and $W\sO_X$ on $X$ by
\[
W\sO_X(U):=W(\sO_X(U)),\ \text{and}\ W_n\sO_X(U):=W_n(\sO_X(U)).
\]
Then we can define $V$, $R$, $F$, $\var{W}_n\sO_X$, and $\var{W}\sO_X$ in a way similar to the case of rings.
We denote the ringed space $(X,W_n\sO_X)$ by $W_nX$, and we can see that this is actually a scheme.
For a morphism $\pi \colon Y \to X$ of schemes, a morphism $W_nY \to W_nX$ of schemes is induced and denoted by $W_n\pi$.
\end{defn}

\begin{rmk}
\label{rmk:ghost component}
We define a ring homomorphism ${\cred w}$ by
\[
{\cred w} \colon W(A) \rightarrow \prod_{\Z_{\geq 0}} A; (a_{0}, \ldots, a_{n}, \ldots) \mapsto ({\cred w}_{0}(a_{0}), \ldots, {\cred w}_{n}(a_{0}, \ldots, a_{n}), \ldots), 
\]
where $\prod_{\Z_{\geq 0}} A$ is the product in the category of rings.
The map ${\cred w}$ is injective if $p \in A$ is a non-zero divisor.
The element ${\cred w}_{n}(a_{0}, \ldots, a_{n})$ (resp.\,the element ${\cred w}(a_{0}, \ldots a_{n}, \ldots)$) is called the ghost component (resp.\,the vector of ghost components) of $(a_{0}, \ldots a_{n}, \ldots) \in W(A)$.
\end{rmk}

\begin{lem}\label{lem:finiteness}
Let $R$ be an $F$-finite ring of positive characteristic ${\cred p}$ and $n$ a positive integer.
Then $F \colon W_n(R) \to F_*W_n(R)$ is a finite $W_n(R)$-module homomorphism.
\end{lem}

\begin{proof}
Let $v_1,\ldots,v_d$ be a generator of $F^n_*R$ over $R$.
We prove that a set
\[
\{V^{r-1}[v_i] \mid 1 \leq i \leq d,\ 0 \leq r \leq n-1\}
\]
generates $F_*W_n(R)$ as a $W_n(R)$-module.
 
We take a positive integer $r \in [1, n]$ and $a \in R$. 
Since $v_1,\ldots,v_d$ be a generator of $F^r_*R$ over $R$, we can take $a_i\in R$ so that $a=\sum a_i^{p^r} v_i$.
Then the $i$-th component of $\sum [a_i^p]V^{r-1}[v_i]$ 
is equal to zero for $i<r$ and is equal to $a$, and we conclude the assertion.
\end{proof}

\subsection{Quasi-$F$-splitting}
In this subsection, we gather basic properties of quasi-$F$-splitting.

\begin{defn}[\cite{Yobuko2}]
\label{defn:quasi-F-split}
Let $X$ be an $\F_p$-scheme.
For a positive integer $n$, we say that $X$ is \textit{$n$-quasi-$F$-split} if 
there exists a $W_n\sO_X$-homomorphism $\phi \colon F_*W_n\sO_X \to \sO_X$ satisfying the following commutative diagram
\[
\xymatrix{
W_n\sO_X \ar[r]^-{F} \ar[d]_-{R^{n-1}} & F_* W_n \ar@{-->}[ld]^{\exists \phi} \sO_X \\
\sO_X.
}
\]

We define the \emph{quasi-$F$-split height} $\sht(X)$ by the infimum of positive integers $n$ such that $X$ is $n$-quasi-$F$-split.
If such an $n$ does not exist, then we set $\sht(X)=\infty$.
We say that $X$ is \emph{quasi-$F$-split} when $\sht(X)$ is finite.
\end{defn}

\begin{rmk}
If $\sht(X)=1$, then we simply say $X$ is \emph{$F$-split}.
By definition, $X$ is $F$-split if and only if the Frobenius map
$\sO_X \to F_*\sO_X$
splits as an $\sO_X$-module homomorphism.
\end{rmk}

\begin{defn}\label{defn:Q}
Let $X$ be an $\F_p$-scheme.
We define a $W_n\sO_X$-module $Q_{X,n}$ by the pushout of the diagram
\[
\xymatrix{
W_n\sO_X \ar[r]^-{F} \ar[d]_{R^{n-1}} & F_*W_n\sO_X \ar[d]  \\
\sO_X \ar[r] & Q_{X,n} \pushoutcorner
}
\]
as $W_n\sO_X$-modules.
Then we have $Q_{X,n} \cong F_*\var{W}_n\sO_X$.
We note that $Q_{X,1}$ is $F_*\sO_X$.
Then $Q_{X,n}$ has an $\sO_X$-module structure 
{\cred
given by
\[
a \cdot F_{\ast} \alpha := [a] F_\ast\alpha = F_\ast ([a^p] \alpha)
\]
for $a \in \sO_X$ and $\alpha \in W_{n} \sO_{X}$, and}
the lower horizontal map in the above diagram 
\[
\sO_X \to Q_{X,n} \ ;\ a \mapsto {\cred F_*} [a^p],
\]
can be seen as an $\sO_X$-module homomorphism.
{\cred The $F$-action on $W_n\sO_X$ are induced to $Q_{X,n}$, which is also denoted by $F$.
Also, $V\colon F_\ast W_n \sO_X \rightarrow W_{n+1} \sO_X$ induces the {\cred $\sO_{X}$-module homomorphism}
\[
F_{\ast} Q_{X,n} \rightarrow Q_{X.n+1},
\]
which is also denoted by $V$.
}

Next, we define $Q_{X}$ by the pushout of the diagram
\[
\xymatrix{
W\sO_X \ar[r]^-{F} \ar[d] & F_*W\sO_X \ar[d] \\
\sO_X \ar[r]& Q_X \pushoutcorner 
}
\]
as $W\sO_X$-modules.
Then we have $Q_X \cong F_*\var{W}\sO_X$, and the lower horizontal map can be seen as $\sO_X$-module homomorphism in a way similar to above.
In addition, the $V$-action and $F$-action on $W\sO_X$ are induce to $Q_{X}$, which are also denoted by $V$ and $F$.

When $X$ is an affine scheme $X=\Spec R$, we denote $Q_{X,n}$ and $Q_X$ by $Q_{R,n}$ and $Q_{R}$, respectively.
\end{defn}

\begin{prop}\label{prop:witt-bar description}
Let $X$ be an $\F_p$-scheme and $n$ a positive integer.
Then $X$ is $n$-quasi-$F$-split if and only if the map
\[
\sO_X \to Q_{X,n}\ ;\ a \mapsto {\cred F_*} [a^p]
\]
splits as an $\sO_X$-module homomorphism.
\end{prop}

\begin{proof}
The assertion follows from the universal property of a pushout square.
\end{proof}

\begin{rmk}\label{rmk:evaluation}
Let $X$ be an $\F_p$-scheme and $n$ a positive integer.
Then $X$ is $n$-quasi-$F$-split if and only if the evaluation map
\[
\Hom_{W_nX}(F_*W_n\sO_X,\sO_X) \to H^0(X,\sO_X)
\]
is surjective.
By Proposition \ref{prop:witt-bar description}, they are also equivalent to the surjectivity of the evaluation map
\[
\Hom_{X} (Q_{X,n},\sO_X) \to H^0(X, \sO_X).
\]
\end{rmk}

\begin{prop}\label{prop:exact sequence of Q_n's}
Let $X$ be an $\mathbb{F}_p$-scheme and $n$ a positive integer.
Then we have the following exact sequence of $\sO_X$-modules
\[
0 \to F^{n}_*(F_*\sO_X/\sO_X) \to Q_{X,n+1} \to Q_{X,n}  \to 0.
\]
\end{prop}
\begin{proof}
By the universality of the pushout $Q_{X,n+1}$, we have a surjective $\sO_X$-module homomorphism $Q_{X,n+1}\to Q_{X,n}$.
We note that this homomorphism coincides with the map $\var{W}_{n+1}\sO_X\to \var{W}_{n}\sO_X$ induced by the restriction $R\colon W_{n+1}\sO_X\to W_{n}\sO_X$.
We denote the map $Q_{X,n+1}\to Q_{X,n}$ by $\pi$.
We have the following commutative diagram
\[
\xymatrix{
 0\ar[r]&\sO_X\ar[r]\ar[d]^{=} &Q_{X,n+1}\ar[r]\ar@{->>}[d]^{\pi}& F_{*}W_{n+1}\sO_X/W_{n+1}\sO_X\ar[r]\ar@{->>}[d]^{R}& 0\\
 0\ar[r]&\sO_X\ar[r] & Q_{X,n}\ar[r]& F_{*}W_{n}\sO_X/W_{n}\sO_X\ar[r]& 0,\\
}
\]
where the horizontal sequences are exact and the right vertical map is induced by the restriction.
By the snake lemma, we have \[\mathrm{Ker}(\pi\colon Q_{X,n+1}\to Q_{X,n})=\mathrm{Ker}(R\colon F_{*}W_{n+1}\sO_X/W_{n+1}\sO_X\to F_{*}W_{n}\sO_X/W_{n}\sO_X).\]
Using an exact sequence
\[
0\to F_{*}^{n}\sO_X\overset{V^n}{\to} W_{n+1}\sO_X \overset{R}{\to} W_{n}\sO_X\to 0
\]
and the snake lemma again, we have
\[
\mathrm{Ker}(R\colon F_{*}W_{n+1}\sO_X/W_{n+1}\sO_{X} \to F_{*}W_{n}\sO_X/W_{n}\sO_X) \cong F^{n}_*(F_*\sO_X/\sO_X),
\]
as desired.
\end{proof}

\begin{prop}\label{prop:multigraded ring versus local ring}
Let $S$ be an $F$-finite $\mathbb{Z}_{\geq 0}$-graded ring such that $S_0$ is a field $k$ of positive characteristic. 
Let $\m:=\bigoplus_{h \neq 0} S_h$.
Then we have $\sht(S)=\sht(S_\m)$.
\end{prop}

\begin{proof}
By \cite[Proposition 7.1]{KTTWYY2}, we can define a graded structure on $W_n(S)$.
Therefore, we can see that $\Hom_{W_n(S)}(F_*W_n(S),S)$ has a graded $W_n(S)$-module structure.
Thus, the surjectivity of $\Hom_{W_n(S)}(F_*W_n(S),S) \to S$ is equivalent to that of $\Hom_{W_n(S_\m)}(F_*W_n(S_\m),S_\m) \to S_\m$, and we conclude that $\sht(S)=\sht(S_\m)$ by Remark \ref{rmk:evaluation}.
\end{proof}

\subsection{Regular sequence and colon ideals}
The following lemma is well-known, but we include the proof for the convenience of the reader.
We need this lemma for the proof of Theorem \ref{Intro:thm:Fedder's criterion}.

\begin{lem}\label{lem:colon rule for regular sequence}
Let $(R,\m)$ be a Noetherian local ring.
Let $f_1,\ldots,f_m$ be a regular sequence {\cred in $\m$}, $f:=f_1 \cdots f_m$, $I:=(f_1,\ldots,f_m)$, and $I_s:=(f_1^s,\ldots,f_m^s)$.
Then, for positive integers $s,r\in\Z_{>0}$ with $s>r$, we have
\begin{itemize}
    \item[\textup{(1)}] $(I_s \colon I)=f^{s-1}R+I_s$,
    \item[\textup{(2)}] $(I_s \colon f^{s-r})=I_r$, and
    \item[\textup{(3)}] $(I_s \colon f^{s-r}I)=((I_s \colon I) \colon f^{s-r})=(I_r \colon I)$.
\end{itemize}
\end{lem}

\begin{proof}
First, we prove (1).
The inclusion $f^{s-1}R+I_s \subseteq (I_s \colon I)$ is clear.
We take an element $x \in R$ such that $xI \subseteq I_s$.
{\cred
Note that $\overline{f_1}$ is a regular element in $R/(f_2^s, \ldots, f_m^s).$
Since we have $xf_1 \in I_s$, we have $\overline{x} \overline{f_1} \in (\overline{f_1}^s)$ in $R/(f_2^s, \ldots, f_m^s).$
Therefore, we have $\overline{x} \in \overline{f_1}^{s-1}$, that means, $x\in I_s + f_1^{s-1}R$.}
Replacing $x$ with another representative of $x+ I_{s}$, we may assume that $x =f_1^{s-1}x_1$ for some ${\cred x_1 \in R}$.
{\cred Since $xf_2 \in I_s$, we have $\overline{f_1}^{s-1}\overline{x_1} \overline{f_2} \in (\overline{f_1}^s)$ in $R/(f_2^s, \ldots, f_m^s)$. Therefore, we have $x_1 f_2 \in (f_1, f_2^s, \ldots, f_m^s)$. On the other hand, since we have $\overline{x_1} \overline{f_2} \in (\overline{f_2}^s)$ in $R/(f_1, f_3^s, \ldots, f_m^s)$, we have $x_1 \in (f_1, f_2^{s-1}) + I_s$.}
Thus replacing $x$ with another representative of $x+I_{s}$, we may assume that $x=f_1^{s-1}f_2^{s-1}x_2$ for some $x_2$.
Repeating this procedure, we obtain $x \in f^{s-1}R+I_s$.

Next, we prove (2). 
The inclusion $I_r \subseteq (I_s \colon f^{s-r})$ is clear.
Let $J:=(f_1,\ldots,f_{m-1})$ and $J_{s}:=(f_1^{s}, \ldots, f_{m-1}^{s})$.
If $xf^{s-r} \in I_s$, then $x (f/f_m)^{s-r} \in (J_s,f_m^{r})$, and thus $x \in (J_s \colon (f/f_m)^{r-s})$ mod $f_m^rR$.
By induction on $m$, we obtain $x \in (f_m^r,J_r)=I_r$. 

Finally, we prove (3).
The first equation is easy.
If $xf^{s-r}I \subseteq I_s$, then $xI \subseteq (I_s \colon f^{s-r})=I_r$ by (2), and thus $x \in (I_r \colon I)$.
If $xI \subseteq I_r$, then $xf^{s-r}I \subseteq f^{s-r}I_r \subseteq I_s$.
Thus, we conclude the second equality.
\end{proof}

\section{Structure of $\Hom_{S}(Q_{S,n}, \sO_S)$}\label{Section:Construction of splitting maps}
In this section, we define a map
\[
\Delta_W \colon \var{W}(R) \to \var{W}(R)/[F(R)]
\]
for a regular local ring $R$, which is a group homomorphism satisfying a variant of the Leibniz rule (Proposition \ref{prop:rule for delta map}).
Moreover, we prove a formula related to $\Delta_W$ (Theorem \ref{Intro:thm:delta formula}) that is essential for the proof of Fedder type criteria for quasi-$F$-splitting.
{\cred They are required to construct a splitting of $R$-module homomorphism}
\[
{\cred 
V^{n-1} \colon 
F^{n-1}_*(F_*R/R) \to Q_{R,n},}
\]
{\cred and, consequently, to obtain a description of $\Hom_R(Q_{R,n},R)$.}

\subsection{Construction of splitting maps}
In this subsection, we define the maps $\Delta_W$ and $\Delta_n$ for $n\in\Z_{>0}$.
When $R$ is a localization of a polynomial ring, they are defined using the monomial decomposition, and they measure how close a given element is to monomials.
In the general setting, we define \emph{$p$-monomials} and \emph{$p$-monomial decompositions}, {\cred which provide an alternative to the usual decompositions}.
Although the usual monomial decomposition is unique, a $p$-monomial decomposition is not unique.
This non-uniqueness plays the essential role when we prove that $\Delta_n$ induces a splitting of
\[
{\cred 
V^{n-1} \colon 
F^{n-1}_*(F_*R/R) \to Q_{R,n},}
\]
as $R$-modules.

\begin{lem}\label{lem:standard basis}
Let $(R,\m,k)$ be an $F$-finite regular local ring of characteristic $p>0$ with  residue field $k$.
Then there exist elements $x_1,\ldots,x_N \in R$ such that 
\[
\{F_*{\cred (}x_1^{i_1} \cdots x_N^{i_N} {\cred )} \mid 0 \leq i_1, \ldots ,i_N \leq p-1 \}
\]
is a basis of $F_*R$ over $R$.
\end{lem}

\begin{proof}
Let $x_1,\ldots,x_n$ be a regular system of parameters of $R$.
Since $R$ is $F$-finite, so is $k$.
Thus, we can take a finite $p$-basis $x'_{n+1},\ldots,x'_N \in k$.
{\cred We take a lift $x_i \in R$ of $x_i' \in k$ for $n+1 \leq i \leq N$.}
Since $R$ is $F$-finite, we may assume that $R$ is a formal power series $k[[x_1,\ldots,x_n]]$ over $k$.
Then $\{x_{n+1}^{i_{n+1}/p} \cdots x_N^{i_N/p} \mid 0 \leq i_{n+1},\ldots,i_N \leq p-1 \}$ is a basis of $k^{1/p}[[x_1,\ldots,x_n]]$ over $k[[x_{1},\ldots,x_n]]$.
Furthermore, $\{x_{1}^{i_{1}/p} \cdots x_n^{i_n/p} \mid 0 \leq i_{1},\ldots,i_n \leq p-1 \}$ is a basis of $k^{1/p}[[x_1^{1/p},\ldots,x_n^{1/p}]]$ over $k^{1/p}[[x_1,\ldots,x_n]]$.
Therefore, the assertion holds.
\end{proof}

\begin{eg}
When $R=k[x_1,\ldots,x_N]$ and $k$ is a perfect field, Lemma \ref{lem:standard basis} asserts that
\[
\{F_* {\cred (}x_1^{i_1}\cdots x_N^{i_N} {\cred )} \mid 0 \leq i_1,\ldots,i_N \leq p-1 \}
\]
is a basis of $F_*R$ over $R$.
\end{eg}

\begin{conv}
\label{conv:basis}
Throughout this section, we fix an $F$-finite regular local ring $(R,\m,k)$  of characteristic $p>0$ with residue field $k$. 
{\cred We take $x_1, \ldots, x_N$ as in Lemma \ref{lem:standard basis}.
We denote the elements of 
\[
\{ x_1^{i_1}\cdots x_N^{i_N}  \mid 0 \leq i_1,\ldots,i_N \leq p-1 \}
\]
by $v_1, \ldots, v_d$, i.e.,\ $F_* v_1, \ldots, F_* v_d$ is a basis of $F_*R$ over $R.$
}
\end{conv}

\begin{defn}\label{definition:monomial}
\ 
\begin{enumerate}
    \item[\textup{(1)}] An element $a \in R$ is called \emph{$p$-monomial} if there exists $a' \in R$ and $i\in \{1,\ldots, d\}$ such that $a=a'^pv_i$.
    \item[\textup{(2)}] A decomposition $a = a_1 + \cdots + a_r$ of $a\in R$ in $R$ is called \emph{$p$-monomial decomposition} if $a_i\in R$ are all $p$-monomials.
\end{enumerate}
\end{defn}
    {\cred Note that the notion of $p$-monomial (decomposition) depends on the choice of a basis $F_*v_1, \ldots, F_*v_d$ in Convention \ref{conv:basis}.}

\begin{rmk}\,
\label{rmk:p-monomial}
\begin{enumerate}
    \item[\textup{(1)}] A product of $p$-monomials is a $p$-monomial. Indeed, we have
    \[
    a^pv_i \cdot b^pv_j=(ab)^pv_iv_j.
    \]
    Furthermore, by the choice of basis (see Lemma \ref{lem:standard basis}), $v_iv_j$ is also a $p$-monomial.
    \item[\textup{(2)}] All elements of $R$ have a $p$-monomial decomposition.
    In fact, for $a \in R$, the basis expansion 
{\cred $F_*a=\sum a_i F_*v_i$ gives a $p$-monomial decomposition $a = \sum a_i^p v_i$.}
    \item[\textup{(3)}] We note that $p$-monomial decompositions are not unique. 
    Indeed, $0=v_i-v_i=v_i+(-1)^pv_i$ is a $p$-monomial decomposition.
    {\cred
    \item[\textup{(4)}]
    Since $v_i =1$ for some $i$, $a^p$ is a $p$-monomial for any $a \in R$.
    }
\end{enumerate}
\end{rmk}

\begin{lem}\label{lem:well-definedness for delta}
\ 
\begin{itemize}
    \item[\textup{(1)}] Let $0=\sum a_i$ be a $p$-monomial decomposition of $0$. Then $\sum[a_i]=0$ in $\var{W}(R)$.
    \item[\textup{(2)}] Let $\alpha \in \var{W}(R)$.
    Then there uniquely exists $\alpha' \in \var{W}(R)$ such that $\alpha-\alpha'$ is a sum of Teichm\"uller lifts of $p$-monomials and $\alpha'$ is contained in the image of $V$.
\end{itemize}
\end{lem}

\begin{proof}
First, we prove (1).
Since every $a_i$ is a $p$-monomial, we can write as $a_i={a'}_i^pv_{r_i}$.
Since $\sum a_i=0$, the coefficients of $v_i$ {\cred are} zero for all $i$. Thus, we may assume that $v:=v_{r_i}=v_{r_j}$ for all $i$ {\cred and $j$}, and we have $\sum {a'}_i=0$.
By the $R$-module structure of $F_*\var{W}(R) \cong Q_{R,n}$, we have
{\cred
\[
F_* \left(\sum [{a'}_i^p v]\right) =\sum F_* ([{a'}^p_i][v])=\sum (a_i'F_*[v]) =(\sum {a'}_i) F_\ast[v]=0,
\]
}
as desired.

Next, we prove (2). 
The zero-th component of $\alpha$ is denoted by $a \in R$.
We take a $p$-monomial decomposition $a=\sum a_i$.
We define $\alpha' \in \var{W}(R)$ by 
\[
\alpha=\sum [a_i]+\alpha'.
\]
Since the zero-th component of $\sum [a_i]$ is $a$, it follows that $\alpha'\in \mathrm{Im}(V)$.
We take $\alpha'_1, \alpha'_2 \in \var{W}(R)\cap\mathrm{Im}(V)$ 
such that $\alpha-\alpha'_i$ is a sum of Teichm\"uller lifts of $p$-monomials for $i\in\{1,2\}$.
Then $\alpha'_1-\alpha'_2$ is a sum of Teichm\"uller lifts of $p$-monomials.
If $\alpha'_1-\alpha'_2=\sum [a_i]$ for $p$-monomials $a_i$, then we have $\sum a_i=0$.
By (1), we have $\sum [a_i]=0$, and thus $\alpha_1'=\alpha_2'$.
Thus, we obtain (2).
\end{proof}

\begin{defn}\label{definition:delta map}
The cokernel of the {\cred additive} map
\[
R \to \var{W}(R)\ ;\ a \mapsto [a^p]
\]
is denoted by $\var{W}(R)/[F(R)]$.
We define the map
\[
\Delta_{W} \colon \var{W}(R) \to \var{W}(R)/[F(R)]
\]
as follows:

Let $\alpha \in \var{W}(R)$.
By Lemma \ref{lem:well-definedness for delta}, there uniquely exists $\alpha' \in \var{W}(R)$ such that $\alpha-\alpha'$ is a sum of Teichm\"uller lifts of $p$-monomials and $\alpha'\in \mathrm{Im}V$.
Then we can define the map
\[
\var{W}(R) \to \mathrm{Im}(V)\ ;\ \alpha \mapsto \alpha'.
\]
By the construction of $\var{W}(R)$ and the map {\cred $V:\var{W}(R) \to \var{W}(R)$}, we have $\var{W}(R)/[F(R)] \cong \mathrm{Im}(V)$.
{\cred We also denote the induced injective additive map 
\begin{equation}
\label{eqn:VfromW/F}
\var{W} (R)/[F(R)] \rightarrow \var{W} (R)
\end{equation}
by $V$.}
Thus, we obtain a map
\[
\Delta_W \colon \var{W}(R) \to \var{W}(R)/[F(R)].
\]
Equivalently, $\Delta_W(\alpha) \in \var{W}(R)/[F(R)]$ is the unique element so that $\alpha-V \Delta_W(\alpha)$ is the sum of Teichm\"uller lifts of $p$-monomials.

Since $a^p$ is a $p$-monomial {\cred for $a\in R$}, it follows that $\alpha+[a^p]-V\Delta_W(\alpha)$ is a sum of Teichm\"uller lifts of $p$-monomials.
Therefore, we have $\Delta_W(\alpha+[a^p])=\Delta_W(\alpha)$, and $\Delta_W$ induces a map
\[
\var{\Delta}_W \colon \var{W}(R)/[F(R)] \to \var{W}(R)/[F(R)].
\]

Next, we define a map $\Delta$ by
\[
\Delta \colon R \to \var{W}(R)/[F(R)]\ ;\ a \mapsto \Delta_W([a]).
\]

Finally, we define a map 
\[
\Delta_r \colon R \to R/F(R)
\]
by  
\[
\Delta_{r}(a) = ({\var{\Delta}}^r_W([a]))_{0},
\]
where ${\var{\Delta}}^r_W$ denotes the $r$-times composition of $\var{\Delta}_W$ and $({\var{\Delta}}^r_W([a]))_{0}$ denotes the zero-th component of ${\var{\Delta}}^r_W([a])$.
\end{defn}

\begin{prop}\label{prop:rule for delta map}
The following hold.
\begin{itemize}
    \item[\textup{(1)}] $\Delta_W(V\alpha)=\alpha$, $\Delta(a)=0$, {\cred and $\Delta_r (a) =0$} for $\alpha \in \var{W}(R)$ and a $p$-monomial $a \in R$,
    \item[\textup{(2)}] $\Delta_W(\alpha+\beta)=\Delta_W(\alpha)+\Delta_W(\beta)$ and $\Delta_W(\alpha \beta)={\cred (} F\alpha {\cred )} \Delta_W(\beta)+ {\cred (} F\beta {\cred )} \Delta_W(\alpha)$ for $\alpha, \beta \in \var{W}(R)$,
    \item[\textup{(3)}] $\Delta(ab)=[a^p]\Delta(b)+[b^p]\Delta(a)$ for $a,b \in R$, and
    \item[\textup{(4)}] $\Delta_r(ab)=a^{p^r}\Delta_{\cred r}(b)+b^{p^r}\Delta_{\cred r}(a)$ for $a,b \in R$ and $r \geq 1$.
\end{itemize}
\end{prop}

\begin{proof}
First, (1) follows from the definition of $\Delta_W,$ {\cred $\Delta,$ and $\Delta_r$}.

Next, we prove (2).
The definition of $\Delta_W$ shows that $\alpha-V\Delta_W(\alpha)$ and $\beta-V\Delta_W(\beta)$ are both sums of Teichm\"uller lifts of $p$-monomials.
Thus, so is
\[
\alpha-V\Delta_W(\alpha)+\beta-V\Delta_W(\beta)\ 
\cred{ = (\alpha + \beta) - V(\Delta_{W} (\alpha) + \Delta_{W}(\beta))}.
\]
By the construction of $\Delta_W$, we have
\[
\Delta_W(\alpha+\beta)=\Delta_W(\alpha)+\Delta_W(\beta).
\]
Furthermore, since the product of $p$-monomials is also a $p$-monomial, 
\[
(\alpha-V\Delta_W(\alpha))(\beta-V\Delta_W(\beta))=\alpha \beta -
V({\cred (} F\alpha {\cred )} \Delta_W(\beta)+ {\cred (}F\beta {\cred )} \Delta_W(\alpha))+V \Delta_W(\alpha) V\Delta_W(\beta)
\]
is a sum of the Teichm\"uller lifts of $p$-monomials.
Since we have
\[
V\Delta_W(\alpha) V\Delta_W(\beta)=V(\Delta_W(\alpha) FV \Delta_W(\beta))=pV(\Delta_W(\alpha) \Delta_W(\beta))=0
\]
in $\var{W}(R)/[F(R)]$, we have
\[
\Delta_W(\alpha \beta)= {\cred (} F\alpha {\cred )} \Delta_W(\beta)+ {\cred (}F\beta {\cred )}\Delta_W(\alpha).
\]
Therefore, we obtain (2).

Thirdly, we prove the assertion (3).
We take $a,b \in R$.
We recall that $\Delta(a)=\Delta_W([a])$ and $\Delta(b)=\Delta_W([b])$.
Then, (2) shows that
\[
\Delta(ab)=\Delta_W([ab])=[a^p]\Delta_W([b])+[b^p]\Delta_W([a])=[a^p]\Delta(b)+[b^p]\Delta(a),
\]
and we obtain (3).

Finally, we prove (4).
Let $a,b \in R$.
First, we compute $\var{\Delta}_W^r([ab])$.
We recall that the zero-th component of $\var{\Delta}_W^r([ab])$ is $\Delta_r(ab)$.
By (3), we have
\[
\var{\Delta}_W([ab])=[a^p]\var{\Delta}_W([b])+[b^p]\var{\Delta}_W([a]),
\]
and by (2), we have
\[
\var{\Delta}_W^2([ab])=\var{\Delta}_W([a^p]\var{\Delta}_W([b]))+\var{\Delta}_W([b^p]\var{\Delta}_W([a])).
\]
By (1) and (2), we have
\[
\var{\Delta}_W^2([ab])=[a^{p^2}]\var{\Delta}_W^2([b])+[b^{p^2}]\var{\Delta}_W^2([a]),
\]
since ${\cred a^p}$ and ${\cred b^p}$ are $p$-monomials {\cred by Remark \ref{rmk:p-monomial}}.
Repeating these arguments, we obtain
\[
\var{\Delta}_W^r([ab])=[a^{p^r}]\var{\Delta}_W^r([b])+[a^{p^r}]\var{\Delta}_W^r([a]).
\]
Comparing the zero-th components, we have
\[
\Delta_r(ab)=a^{p^r}\Delta_r({\cred b})+b^{p^r}\Delta_r({\cred a}),
\]
as desired.
\end{proof}

\begin{rmk}\label{eg:explicit delta}
Let $f \in R$, and let $f = f_1 + \cdots + f_m$ be a $p$-monomial decomposition.
{\cred Since $[f] - \sum [f_i] \in \mathrm{Im}(V)$, there exist elements $\delta_1(f), \delta_2(f), \ldots \in R$ such that
\[
[f] = \sum [f_i] + (0, \delta_1(f), \delta_2(f), \ldots).
\]
Here, note that the tuple $(\delta_1(f), \delta_2(f), \ldots) $  depends on the choice of $p$-monomial decomposition of $f\in R$.
It follows that
\[
\Delta(f) = (\delta_1(f), \delta_2(f), \ldots)
\]
in $\var{W}(R) / [F(R)]$, and in particular, we have $\Delta_1(f) = \delta_1(f)$.
Next, we compute
\[
{\var\Delta}_W^2([f]) = {\var\Delta}_W((\delta_1(f), \delta_2(f), \ldots)) = {\var\Delta}_W([\delta_1(f)]) + (\delta_2(f), \delta_3(f), \ldots).
\]
By iterating this process, we obtain
\[
{\var\Delta}_W^s([f]) = {\var\Delta}_W^{s-1}([\delta_1(f)]) + {\var\Delta}_W^{s-2}([\delta_2(f)]) + \cdots + {\var\Delta}_W([\delta_{s-1}(f)]) + (\delta_s(f), \delta_{s+1}(f), \ldots).
\]
Therefore, we have
\[
\Delta_s(f) = \sum_{r=1}^{s-1} \Delta_r \circ \delta_{s-r}(f) + \delta_s(f),
\]
as in (\ref{eqn:Deltadelta}).
}
By the additive structure of the ring of Witt vectors, we have
\[
\delta_1(f) = \sum_{\substack{0 \leq \alpha_1, \ldots, \alpha_m \leq p-1 \\ \alpha_1 + \cdots + \alpha_m = p}} \frac{1}{p} \binom{p}{\alpha_1, \ldots ,\alpha_m} f_1^{\alpha_1} \cdots f_m^{\alpha_m}.
\]
\end{rmk}

\begin{prop}\label{prop:delta gives a splitting}
We define a{\cred n} {\cred additive} map $\widetilde{\sigma}_n \colon \var{W}(R) \to R/F(R)$ by 
\[
\xymatrix{
\var{W}(R) \ar[r] & \var{W}(R)/[F(R)] \ar[r]^-{\var{\Delta}_{W}^{n-1}} & \var{W}(R)/[F(R)] \ar[r] & R/F(R),
}
\]
where the first map is the natural surjection and the last map is the map induced by the restriction map.
Then $\widetilde{\sigma}_n$ induces an $R$-module homomorphism
\[
\sigma_n \colon Q_{R,n}\to F^{n-1}_*(F_*R/R)
\]
which satisfies
\begin{align}\label{equation:description of sigma}
     \sigma_n({\cred F_*} (a_0,a_1,\ldots,a_{n-1}))= {\cred F^n_*(}\Delta_{n-1}(a_0)+\Delta_{n-2}(a_1)+\cdots+a_{n-1} {\cred)}
\end{align}
{\cred for $a_0, \ldots, a_{n-1} \in R$.}
In particular, $\sigma_n$ gives a splitting of the {\cred injective} $R$-module homomorphism
\[
V^{n-1} \colon F^{n-1}_*(F_*R/R) \to Q_{R,n} \ ;\ {\cred F^n_*} a \mapsto {\cred F_* (} V^{n-1}[a] {\cred )}.
\]
\end{prop}

\begin{proof}
Let $a\in R$ and $m\in\Z_{>0}$ such that $m \geq n$.
Since $V^{m}[a] \in \mathrm{Im}(V^{n})$, Proposition \ref{prop:rule for delta map} (1) shows that $\var{\Delta}_{W}^{n-1}(V^{m}[a])=V^{m-(n-1)}[a] \in \mathrm{Im}(V)$.
Thus, we have $\widetilde{\sigma}_n(V^{m}[a])=0$, and $\widetilde{\sigma}_n$ induces a{\cred n} {\cred additive} map
\[
\sigma_n \colon \var{W}_n(R) \to R/F(R).
\]
By Proposition \ref{prop:rule for delta map} {\cred (1) and (2)} {\cred and Remark \ref{rmk:p-monomial} (4)},
we have
\[
\var{\Delta}_{W}^{n-1}([a^p] \alpha)=a^{p^{n}}\var{\Delta}_{W}^{n-1}(\alpha)
\]
and $\var{\Delta}_{W}^{n-1}(\alpha+\beta)=\var{\Delta}_{W}^{n-1}(\alpha)+\var{\Delta}_{W}^{n-1}(\beta)$
{\cred for $a\in R$ and $\alpha, \beta \in \overline{W}(R)$}.
Thus, $\sigma_n \colon Q_{R,n} \to F^{n-1}_*(F_*R/R)$ is an $R$-module homomorphism.
Since we have
\begin{align}\label{equation:sigma_n}
\var{\Delta}_{W}^{n-1}(V^{s-1}[a])=\var{\Delta}_{W}^{n-s}([a])    
\end{align}
by Proposition \ref{prop:rule for delta map} for $s \leq n$.
The zero-th component of $\var{\Delta}_{W}^{n-s}([a])$ is $\Delta_{n-s}(a)$ if $s \leq n-1$, and $a$ if $s=n$.
Comparing the zero-th components of (\ref{equation:sigma_n}), we have $\sigma_{n}(V^{s-1}[a])=\Delta_{n-s}(a)$ for $s \leq n-1$ and $\sigma_{n}(V^{n-1}[a])=a$.
Therefore, we obtain (\ref{equation:description of sigma}).
\end{proof}

\subsection{Delta formula}
In the previous subsection, we construct a splitting $\sigma_n$ of $V^{n-1}$.
By the description (\ref{equation:description of sigma}), an explicit description of $\Delta_n$ gives that of $\sigma_n$.
In this subsection, we prove the delta formula (Theorem \ref{thm:delta formula}), which means that all $\Delta_n$ can be computed by only $\Delta_1$.
It is an essential reason why we only need the ring structure of $W_2(R)$ in a Fedder type criterion for quasi-$F$-splitting (Theorem \ref{Intro:thm:Fedder's criterion}).

For the proof of the delta formula, we need the map ${\cred w}$ in Remark \ref{rmk:ghost component}, and ${\cred w}$ behaves well only in mixed characteristic.
Therefore, we first define $\Delta_n$ in mixed characteristic using a lift of the Frobenius map.

\begin{conv}
Throughout this subsection, we set
\[
A:=\Z[X_{f} \mid f \colon \text{$p$-monomial in } R\ ]. 
\]
We define a lift of Frobenius $\phi$ as the ring homomorphism defined by $\phi(X_f)=X_f^p$.
We recall the 
{\cred ring homomorphism}
\[
{\cred w} \colon W(A) \to \prod_{\Z_{ \geq 0}} A
\]
in Remark \ref{rmk:ghost component}.
The $r$-th component of ${\cred w}((a_0,a_1,a_2,\ldots ))$ is
\[
a_0^{p^{r}}+pa_1^{p^{r-1}}+ \cdots p^{r-1}a_{r-1}^p+p^{r}a_{r}.
\]
In addition, ${\cred w}$ is injective since $A$ is $p$-torsion free.
\end{conv}

\begin{defn}
For $x \in A$, we say that $x$ is a \emph{$\phi$-monomial} if $\phi(x)=x^p$.
We say that $x=\sum m_i x_i$ is a \emph{$\phi$-monomial decomposition} if $x_i$ is a $\phi$-monomial and $m_i \in \Z$ for all $i$.
\\
For $\alpha \in W(A)$, we say that $\alpha=\sum m_i \alpha_i$ is a \emph{$\phi$-monomial decomposition} if $\alpha_i$ is the Teichm\"uller lift of a $\phi$-monomial and $m_i \in \Z$ for all $i$.
\end{defn}

\begin{rmk}
An integer $m\in\Z$ is not a $\phi$-monomial in general.
{\cred Indeed, an integer $m$ is $\phi$-monomial if and only if $m^p = m$ (that is, $m=0, m=1,$ or $p \neq 2$ and $m=-1$).}
\end{rmk}

\begin{lem}\label{lem:well-definedness for delta map in mixed characteristic}
\ 
\begin{itemize}
    \item[\textup{(1)}] Let $0=\sum m_ix_i$ be a $\phi$-monomial decomposition of zero{\cred , where $m_i \in \Z$ and $x_i \in A$}. Then $\sum m_i [x_i]=0$.
    \item[\textup{(2)}] Every element of $A$ has a $\phi$-monomial decomposition.
    \item[\textup{(3)}] For $\alpha \in W(A)$, there uniquely exists $\alpha' \in \mathrm{Im}(V)$ such that $\alpha-\alpha'$ has a $\phi$-monomial decomposition.
\end{itemize}
\end{lem}

\begin{proof}
We prove (1).
Since ${\cred w}$ is a ring homomorphism, we have
\begin{align*}
    {\cred w}(\sum m_i [x_i])=& \sum m_i {\cred w}({\cred[}x_i {\cred]}) = \sum m_i (x_i,x_i^p,x_i^{p^2},\ldots) 
    = (\sum m_i x_i, \sum m_i x_i^p,\sum m_i x_i^{p^2},\ldots).
\end{align*}
Since $x_i^{p^r}=\phi^r(x_i)$ and $\phi$ is a ring homomorphism, we obtain
\[
\sum m_i x_i^{p^r}=\sum m_i \phi^r(x_i) =\phi^r(\sum m_i x_i).
\]
Then we have $\sum m_i x_i^{p^r}=0$ since $\sum m_i x_i=0$.
We conclude that $\sum m_i [x_i]=0$ by the injectivity of ${\cred w}$.

(2) follows from the fact that every monomial is a $\phi$-monomial.

Finally, using the argument in Lemma \ref{lem:well-definedness for delta}, (3) follows from (1) and (2). 
\end{proof}

\begin{defn}
We define the map
\[
\Delta_{W} \colon W(A) \to W(A)
\]
as follows:

Let $\alpha \in W(A)$.
By Lemma \ref{lem:well-definedness for delta map in mixed characteristic}, there uniquely exists $\alpha' \in W(A)$ such that $\alpha-\alpha'$ has a $\phi$-monomial decomposition and $\alpha'\in \mathrm{Im}(V)$.
Since $V$ is injective, there uniquely exists an element of $W(A)$ whose image by $V$ is $\alpha'$.
Thus, we can define $\Delta_W(\alpha)$ {\cred $\in W(A)$} satisfying $V  \Delta_W(\alpha)=\alpha'$.
Equivalently, $\Delta_W(\alpha)$ is the unique element {\cred of $W(A)$} satisfying that $\alpha-V \Delta_W(\alpha)$ has a $\phi$-monomial decomposition.

Next, we define a map $\Delta$ by
\[
\Delta \colon A \to W(A)\ ;\ a \mapsto \Delta_W([a]),
\]
and $\Delta_r \colon A \to A$ by the zero-th component of $\Delta_W^r([a])$.
Furthermore, for $a \in A$, we define $\delta_r(a) \in A$ satisfying
\[
\Delta(a)=(\delta_1(a),\delta_2(a),\ldots )
\]
\end{defn}

\begin{lem}\label{lem:inductive rule for delta}
\begin{enumerate}
\item
{\cred $\Delta_{W} \colon W(A) \rightarrow W(A)$ is an additive map and $\Delta_W \circ V=\mathrm{id}_{W(A)}$.}
\item
For all $n \geq 1$, we have
\[
\Delta_n=\sum^{n-1}_{r=1} \Delta_r \circ \delta_{n-r} +\delta_n
\]
\end{enumerate}
\end{lem}

\begin{proof}
{\cred First, we prove (1).
Take $\alpha, \beta \in W(A)$. Then $\alpha - V \Delta_W(\alpha)$ and $\beta - V \Delta_W(\beta)$ have $\phi$-monomial decompositions by definition.
Therefore, the element $(\alpha + \beta) - V(\Delta_W(\alpha) + \Delta_W(\beta))$ also has a $\phi$-monomial decomposition. In particular, we have $\Delta_W(\alpha + \beta) = \Delta_W(\alpha) + \Delta_W(\beta)$, as claimed.
Furthermore, for $\alpha \in W(A)$, since the element $V\alpha - V\alpha = 0$ admits a $\phi$-monomial decomposition, we obtain $\Delta_W(V\alpha) = \alpha$, as claimed.
}

{\cred Next, we prove (2).}
Let $a \in A$.
Then $\Delta_{W}([a])=\Delta(a)=(\delta_1(a),\delta_2(a), \ldots)$.
{\cred Furthermore},  we have
\begin{eqnarray*}
\Delta_{W}^n([a])&=&\Delta_{W}^{n-1} \circ \Delta_{W}([a])=\Delta_{W}^{n-1}((\delta_1(a),\delta_2(a),\ldots)) \\
&=& 
{\cred
\Delta_W^{n-1}([\delta_1(a)])+ \Delta_W^{n-2}\circ \Delta_W((0, \delta_2(a), \ldots)) } \\
&{\cred  =}&
{\cred \Delta_W^{n-1}([\delta_1(a)])+\Delta_{W}^{n-2}(\delta_2(a),\delta_3(a),\ldots)}
\end{eqnarray*}
{\cred by (1).}
Repeating this procedure, we obtain 
\[
\Delta_W^n([a])=\Delta_W^{n-1}([\delta_1(a)])+\Delta_W^{n-2}([\delta_2(a)])+\cdots+\Delta_W([\delta_{n-1}(a)])+[\delta_n(a)].
\]
Therefore, comparing the zero-th component, we have
\[
\Delta_n(a)=\Delta_{n-1} \circ \delta_1(a)+\cdots +\Delta_1 \circ \delta_{n-1}(a)+\delta_n(a).
\]
\end{proof}

\begin{lem}\label{lem:delta formula in mixed characteristic}
For $a \in A$ and $n \geq 1$, we have
\[
\Delta_n(a)=\frac{\Delta_1(a^{p^{n-1}})}{p^{{n-1}}}=\frac{a^{p^{n}}-(a^p-p\Delta_1(a))^{p^{n-1}}}{p^{n}}
\]
\end{lem}

\begin{proof}
First, we prove the latter equality.
Let $a\in A$. By the definition of $\Delta$, we have 
\[
[a]=\sum m_i [a_i]+V  \Delta(a)=\sum m_i [a_i]+(0,\delta_1(a),\delta_2(a),\ldots),
\]
where $m_i$ is an integer and $a_i$ is a $\phi$-monomial for every $i$.
Comparing the $r$-th component of the images by ${\cred w}$, we have
\begin{equation}\label{equation1:delta formula}
a^{p^{r}}=\phi^{r}(a)+p\delta_1(a)^{p^{r-1}}+ \cdots + p^{r-1}\delta_{r-1}(a)^p+p^r\delta_r(a)
\end{equation}
for all $r \geq 1$.
Taking $r=1$, we have 
\[
a^p=\phi(a)+p\delta_1(a)=\phi(a)+p\Delta_1(a){\cred ,}
\]
{\cred where the second equality follows from Lemma \ref{lem:inductive rule for delta} (2).}
Therefore, we obtain
\[
p\Delta_1(a^{p^{n-1}})=a^{p^{n}}-\phi(a)^{p^{n-1}}=a^{p^n}-(a^p-p\Delta_1(a))^{p^{n-1}}, 
\]
and thus the latter equation holds.

Next, we prove the first equality by induction on $n$.
We take $r=n-1$ in (\ref{equation1:delta formula}) and consider its image by $\phi$.
Then we have
\begin{equation}\label{equation2:delta formula}
\phi(a)^{p^{n-1}}=\phi^{n}(a)+p\phi(\delta_1(a)^{p^{n-2}})+ \cdots + p^{n-2}\phi(\delta_{n-2}(a)^p)+p^{n-1}\phi(\delta_{n-1}(a)).
\end{equation}
We take $r=n$ (\ref{equation1:delta formula}).
Taking the difference with (\ref{equation2:delta formula}), 
we have the equality of
\[
a^{p^n}-\phi(a)^{p^{n-1}}=p\Delta_1(a^{p^n-1}),
\]
and 
\begin{align}\label{equation:RHS}
    & p(\delta_1(a)^{p^{n-1}}-\phi(\delta_1(a)^{p^{n-2}}))+\cdots+p^{n-1}(\delta_{n-1}(a)^p-\phi(\delta_{n-1}(a)))+p^n\delta_n(a)\notag \\
    =&  p^2\Delta_1(\delta_1(a)^{p^{n-2}})+\cdots +p^n\Delta_1(\delta_{n-1}(a))+p^n\delta_n(a)\notag \\
    =& p^n(\Delta_{n-1}\circ \delta_1(a)+\cdots +\Delta_1 \circ \delta_{n-1}(a)+\delta_n(a))\notag \\
    =& p^n \Delta_n(a),
\end{align}
where the second and last equations in (\ref{equation:RHS}) follow from the induction hypothesis and Lemma \ref{lem:inductive rule for delta} respectively.
Therefore, we obtain
\[
p\Delta_1(a^{p^n-1})=p^n \Delta_n(a),
\]
as desired.
\end{proof}

\begin{rmk}
A notion similar to $\Delta_1$ has already appeared in \cite[{\cred Definition 2.1 and Remark 2.2}]{BS} as $\delta$, but we remark that they have different signs.
Indeed, we have 
\[
a^p=\phi(a)+p\Delta_1(a)
\]
in the proof of Lemma \ref{lem:delta formula in mixed characteristic}, 
and thus $\Delta_1(a)=-\delta(a)$.
\end{rmk}

\begin{lem}\label{lem:commutativity of delta}
{\cred We define the ring homomorphism
\[
\pi \colon A \rightarrow R\ ;\ X_f \mapsto f,
\]
and the additive map
\[
\var{\pi} \colon A \xrightarrow{\pi} R \to R/F(R).
\]
}
Then we have $\Delta_r \circ \pi=\var{\pi} \circ \Delta_r$ {\cred as a morphism of sets for every $r \geq 1$}.
\end{lem}

\begin{proof}
First, we note that the images of variables in $A$ by $\pi$ are $p$-monomials.
Therefore, the images of monomials in $A$ by $\pi$ are $p$-monomials.
We take $x \in A$.
Then there exists a $\phi$-monomial decomposition $x =\sum m_i x_i$ such that every $x_i$ is a monomial and $m_i$ is an integer.
Then $\pi(x_i)$ is a $p$-monomial, {\cred and} thus $\pi(m_ix_i)=m_i\pi(x_i)$ {\cred $= m_i^p \pi (x_i)$} is also a $p$-monomial.
Therefore, $\pi(x)=\sum \pi(m_i x_i)$ is a $p$-monomial decomposition.
Furthermore, by the definition of $\Delta$, we have
\[
V \Delta(x)=[x]-\sum m_i [x_i],\ \text{and}\ V \Delta(\pi(x))=[\pi(x)]-\sum [\pi(m_i x_i)].
\]
Denoting natural maps $W(A) \to \var{W}(R)$ by $\pi_W$ and $W(A) \to \var{W}(R)/[F(R)]$ by $\var{\pi}_W$, then we have
\begin{align*}
    V \var{\pi}_W(\Delta(x))
    &= {\cred {\pi}_W}(V \Delta(x))=\pi_W([x])-\sum \pi_W(m_i [x_i]) \\
    &=[\pi(x)]-\sum [\pi(m_i x_i)]=V \Delta(\pi(x)).
\end{align*}
{\cred Here, $V$ is the injective additive map defined in (\ref{eqn:VfromW/F}).}
Therefore, we have  $\var{\pi}_W(\Delta(x))=\Delta(\pi(x))$ for all $x \in A$.
Since $\Delta_W(\alpha)=\Delta(a_0)+(a_1,a_2,\ldots)$ for $\alpha=(a_0,a_1,a_2,\ldots)$, we also have $\var{\pi}_W \circ \Delta_W=\Delta_W \circ \pi_W$.
By the construction of $\Delta_r$ it implies the commutativity for $\Delta_r$.
\end{proof}

\begin{thm}[Delta formula (Theorem \ref{Intro:thm:delta formula})]\label{thm:delta formula}
{\cred We have an equality}
$\Delta_n(f)=f^{p^n-p} \Delta_1(f)$
{\cred in $R/F(R)$}
for all
{\cred $n\geq 1$ and}
$f \in R$.
{\cred Note that since the exponent of $f$ is $p$-divisible, the product is well-defined.}
\end{thm}

\begin{proof}
{\cred We may assume $n \geq 2$.}
Let $f \in R$.
Let $\pi$ and $\var{\pi}$ be maps as in Lemma \ref{lem:commutativity of delta}.
Since every element of $R$ has a $p$-monomial decomposition, we have $\pi$ is surjective.
Therefore, we can take a lift $x \in A$ of $f$.
By Lemma \ref{lem:commutativity of delta}, we have $\Delta_n(f)=$ {\cred $\Delta_n (\pi(x)) =$} $\var{\pi}(\Delta_n(x))$.
By Theorem \ref{lem:delta formula in mixed characteristic}, the latter part is
\[
\var{\pi}(\Delta_n(x))=\var{\pi}(\frac{x^{p^{n}}-(x^p-p\Delta_1(x))^{p^{n-1}}}{p^{n}}).
\]
{\cred On the other hand, we have}
\begin{align*}
    &\frac{x^{p^{n}}-(x^p-p\Delta_1(x))^{p^{n-1}}}{p^{n}} \\
    =&x^{p^n-p}\Delta_1(x)-\frac{1}{p^{n-2}}\binom{p^{n-1}}{2}x^{p^n-2p}\Delta_1(x)^{2}+\cdots +(-1)^{p+1}p^{p^{n-1}-n}\Delta_1(x)^{p^{n-1}},
\end{align*}
{\cred where the $p$-order $\nu_p$ of each coefficient (other than the first term) is given by
\[
\nu_p\binom{p^{n-1}}{i} - (n-i) = n - 1- \nu_p (i) -(n-i) =i-1 - \nu_p(i)
\]
for $i\geq 2$ by Kummer's theorem.
Moreover, we have
\[
i-1-\nu_p(i) \geq p^{\nu_p(i)} -1-\nu_p(i) \geq 0,
\]
and the left-hand side is equal to $0$ if and only if $p=2$ and $i=2$.
In this case, the term is given by
\[
-\frac{1}{2^{n-2}} \binom{2^{n-1}}{2} x^{2^n-4}\Delta_1(x)^2
=-(2^{n-1}-1)x^{2^n-4} \Delta_1(x)^2,
\]
whose image by $\overline{\pi}$ is contained in the image of $F$.
Therefore,} the image of $\Delta_{n}(x)$ by $\var{\pi}$ is
\[
\var{\pi}(\Delta_n(x))=\var{\pi}(x^{p^n-p}\Delta_1(x))=f^{p^n-p}\Delta_1(f),
\]
as desired.
{\cred Here, the first equality follows from the above argument, and the second equality follows from Lemma \ref{lem:commutativity of delta}.
Note that for any $a, b \in A$, we have $\overline{\pi}(a^pb) = \overline{\pi} (a)^p \overline{\pi}(b)$ since $\pi$ is a ring homomorphism.
}
\end{proof}

\section{Fedder type criteria for quasi-$F$-splitting}\label{Section:Fedder type criteria for quasi-$F$-splitting}
In this section, we use the following notation.
\begin{conv}\label{conv:regular local ring}
Throughout this section, a pair $(R,\m,k)$ is an $F$-finite regular local ring of characteristic $p>0$ with residue field $k$ and we take {\cred $v_1, \ldots ,v_d \in R$ as in Convention \ref{conv:basis},} where $v_d=(x_1 \cdots x_n)^{p-1}$.
Furthermore, we denote the dual basis by $u_1, \ldots, u_d \in \Hom_R(F_*R,R)$ with respect to ${\cred F_*v_1,\ldots, F_*v_d}$.
We note that $u_d$ is a generator of $\Hom_R(F_*R,R)$ as an $F_*R$-module, 
{\cred since} $u_i =u_d( {\cred F_*( v_d/v_i)} \cdot \underline{\ \ })$ for all $i$.
Moreover, we denote $v_d$ by $v$ and $u_d$ by $u$, for short.
We denote the kernel of $u$ by $v^\perp {\cred \subset F_*R}$.
{\cred For $F_*g \in v^\perp$ and an integer $n \geq 1$, the $R$-module homomorphism $u^{n}(F^n_*(g )\cdot \underline{\ \ }) \colon F^n_* R \to R$ induces an $R$-module homomorphism $u^{n}_g \colon F^{n-1}_*(F_* R / R) \to R$.
For $a \in R / F(R)$, we denote $u_g^n(F^n_* a)$ by $u^{n}(F^n_*(g a))$, by abuse of notation.
We note that for any representative $\tilde{a} \in R$ of $a \in R / F(R)$, we have $u^n(F^n_*(g \tilde{a})) = u^{n}(F^n_*(g a))$.
}
\end{conv}

In the previous section, we study the explicit structure of the splitting of $V^{n-1}$.
By using the splitting $\sigma_{n}$ of the exact sequence
\[
\xymatrix{
0 \ar[r] & F^{n-1}_*(F_*R/R) \ar[r] & Q_{R,n} \ar[r] & Q_{R,n-1} \ar[r] & 0,
}
\]
we can describe the structure of $\Hom_R(Q_{R,n},R)$, inductively, and we obtain the explicit {\cred $R$-module isomorphism}
\[
\Psi_n \colon F_*R \oplus \bigoplus_{2 \leq e \leq n} F^{e-1}_*
{\cred v^\perp}
\to \Hom_R(Q_{R,n},R).
\]
In other words, we describe all {\cred $R$-module} homomorphisms from $Q_{R,n}$ to $R$ by using tuples of elements of $R$, explicitly.
After that, using the description, we prove Theorem \ref{Intro:thm:Fedder's criterion}.

First of all, we prove the following result.

\begin{lem}\label{lem:basis for Witt ring}
$Q_{R,n}$ is a free $R$-module for all $n\in\Z_{>0}$.
\end{lem}

\begin{proof}
By Kunz's theorem, $Q_{R,1}=F_*R$ is free.
Since $R$ is regular, and in particular $F$-split, $F_{*}R/R$ is free.
Now, we can deduce from the exact sequence
\[
\xymatrix{
0 \ar[r] & F^{n}_*(F_*R/R) \ar[r] & Q_{R,n+1} \ar[r] & Q_{R,n} \ar[r] & 0
}
\]
in Proposition \ref{prop:exact sequence of Q_n's} that
$Q_{R,n+1}$ is free inductively.
\end{proof}

\begin{lem}\label{lem:structure of homomorphism for n}
There exists an {\cred $R$-module} isomorphism
\[
\Psi_n \colon F_*R \oplus \bigoplus_{2 \leq e \leq n} F^{e-1}_* {\cred v^\perp}\to \Hom_R(Q_{R,n},R)
\]
for $n\in\Z_{>0}$ {\cred satisfying the following:}
Let $\psi_{(g_1,\ldots,g_n)}:=\Psi_n(({\cred F_*g_1,\ldots,F_*^ng_n}))$
{\cred
for 
\[
(F_*g_1, \ldots, F^n_*g_n) \in  F_*R \oplus \bigoplus_{2 \leq e \leq n} F^{e-1}_*v^\perp \subseteq
\bigoplus_{1 \leq e \leq n} F^{e}_* R.
\]
}
Then we have
\[
\psi_{(g_1,\ldots,g_n)}({\cred F_* (}V^{s-1}[a]))=\sum_{r=0}^{n-s} u^{r+s}(F^{r+s}_*{\cred (} g_{r+s}\Delta_{r}(a) {\cred )})
\]
for $s\in\{1,\ldots, n\}$, where we set $\Delta_0(a):=a$,
{\cred and $u^{r+s}(F^{r+s}_*(g_{r+s}\Delta_{r}(a)))$ for $r \geq 1$ is defined in Convention \ref{conv:regular local ring}.}
\end{lem}

\begin{proof}
We prove the assertion by induction on $n$.
When $n=1$, it follows from the fact that $u$ is a generator of $\Hom_R(F_*R,R)$ as an $F_*R$-module, that is,
the map
\[
\Psi_1 \colon F_*R \to \Hom_R(Q_{R,n},R)\ ;\ g \mapsto \psi_g:=(F_*a \mapsto u(F_* {\cred (}ga {\cred )}))
\]
is an {\cred $R$-module} isomorphism.
We assume $n \geq 2$. We consider the exact sequence
\[
\xymatrix{
0 \ar[r] & F^{n-1}_*(F_*R/R) \ar[r]^-{V^{n-1}} & Q_{R,n} \ar[r]^-{\pi} & Q_{R,n-1} \ar[r] & 0,
}
\]
in Proposition \ref{prop:exact sequence of Q_n's}. 
By Proposition \ref{prop:delta gives a splitting},
the above sequence splits by $\sigma_n \colon Q_{R,n} \to F^{n-1}_*(F_*R/R)$.
Thus, we obtain an {\cred $R$-module} isomorphism
\[
(\pi^*,\sigma_n^*) \colon \Hom_R(Q_{R,n-1},R) \oplus \Hom_R(F^{n-1}_*(F_*R/R),R) \to \Hom_R(Q_{R,n},R).
\]
Using the {\cred $R$-module} isomorphism
\[
F^n_*R \to \Hom_R(F^n_*R,R) \ ;\ F^n_*g \mapsto (F^n_*a \mapsto u^n(F_*^n {\cred (}ga {\cred )})),
\]
we obtain 
\[
F^{n-1}_*{\cred v^\perp} \to \Hom_R(F^{n-1}_*(F_*R/R),R). 
\]
We denote the image of $g$ by the above map by $\var{\psi}_g$.
By definition, the composition of the maps
\[
F^n_*R \rightarrow F^{n-1}_*(F_*{\cred R} /R) \xrightarrow{\var{\psi}_g} R
\]
coincides with $\psi_g$.
By the induction hypothesis, we have an $R$-module isomorphism
\[
\Psi_{n-1} \colon F_*R \oplus \bigoplus_{2 \leq e \leq n-1} F^{e-1}_*{\cred v^\perp} \to \Hom_R(Q_{R,n-1},R)
\]
satisfying the conditions in the statement.
Combining the above isomorphisms, we obtain an $R$-module isomorphism
\[
\Psi_n \colon F_*R \oplus \bigoplus_{2 \leq e \leq n} F^{e-1}_*{\cred v^\perp} \to \Hom_R(Q_{R,n},R).
\]
By the definition of $\Psi_n$, we have
\[
\Psi_n(({\cred F_*g_1,\ldots, F^n_*g_n}))=\pi^*\Psi_{n-1}(({\cred F_*g_1,\ldots, F^{n-1}_*g_{n-1}}))+\sigma_n^*\var{\psi}_{g_n}.
\]
We take $a \in R$ and $s \in \{1,\ldots,n\}$.
When $s=n$, we have $\pi(V^{n-1}[a])=0$, and thus
\[
\Psi_n(({\cred F_*g_1,\ldots, F^n_*g_n}))({\cred F_* (}V^{n-1}[a]{\cred )})=\var{\psi}_{g_n}(\sigma_n(V^{n-1}[a]))=\var{\psi}_{g_{n}}(a)=u^n(F^n_* {\cred (}g_na {\cred )}),
\]
where we note $\sigma_n({\cred F_*}V^{n-1}[a])=a$.
When $s \leq n-1$, we have
\begin{align*}
    &\Psi_n(({\cred F_*g_1,\ldots, F^n_*g_n}))({\cred F_* (} V^{s-1}[a]{\cred )}) \\
    =&\Psi_{n-1}(({\cred F_*g_1,\ldots, F^{n-1}_*g_{n-1}}))({\cred F_*(}V^{s-1}[a]{\cred )})+(u^n(F^n_*{\cred (}g_n \sigma_n({\cred F_* (}V^{s-1}[a]{\cred )}){\cred )})) \\
    =&\sum_{r=0}^{n-1-s} u^{r+s}(F^{r+s}_* {\cred (} g_{r+s}\Delta_{r} (a) {\cred )})+u^n(F^n_*{\cred (}g_n \sigma_n({\cred F_* (}V^{s-1}[a]{\cred )}){\cred )})
\end{align*}
by the induction hypothesis.
By Proposition \ref{prop:delta gives a splitting}, $\sigma_n(V^{s-1}[a])=\Delta_{n-s}(a)$.
Therefore, we have
\[
\Psi_n((g_1,\ldots,g_n))({\cred F_* (}V^{s-1}[a]{\cred )})=\sum_{r=0}^{n-s} u^{r+s}(F^{r+s}_* {\cred (}g_{r+s}\Delta_r(a){\cred )}),
\]
as required.
\end{proof}

\begin{lem}\label{lem:commutativity of the structure morphism}
We have the following commutative diagram
{\cred of $R$-modules}
\[
\xymatrix{
F_*R \oplus \bigoplus_{\substack{2 \leq e \leq n}} F^{e-1}_*{\cred v^\perp} \ar[r]^-{\Psi_n}_-{{\cred \simeq}} \ar@{^{(}->}[d] \ar@{}[rd]|{\circlearrowright} & \Hom_R(Q_{R,n},R) \ar@{^{(}->}[d]^-{\pi^*} \\
F_*R \oplus \bigoplus_{2 \leq e \leq n+1} F^{e-1}_*{\cred v^\perp} \ar[r]^-{\Psi_{n+1}}_-{{\cred \simeq}}  & \Hom_R(Q_{R,n+1},R),
}
\]
where the left vertical map is defined by 
\[
({\cred F_* g_1,\ldots, F^n_*g_n}) \mapsto ({\cred  F_* g_1, \ldots, F^n_*g_n},0).
\]
\end{lem}

\begin{proof}
The assertion follows from the construction of $\Psi$.
We note that
\[
\Psi_{n+1}(({\cred F_*g_1,\ldots, F^n_*g_n},0))=\pi^*\Psi_{n}(({\cred F_* g_1,\ldots, F^n_*g_n}))+\sigma_n^*\var{\psi}_0=\pi^*\Psi_{n}(({\cred F_*g_1,\ldots, F^n_*g_n})).
\]
\end{proof}

\begin{defn}\label{defn:structure morphism}
Let $\psi \in \Hom_R(Q_{R},R)$.
We say that $\psi$ is \emph{of rank $n$} if $\psi$ {\cred factors through the restriction map $Q_R \rightarrow Q_{R,n}$.} 
By Lemma \ref{lem:commutativity of the structure morphism}, we obtain a{\cred n $R$-module homomorphism}
\[
\Psi \colon F_*R \oplus \bigoplus_{2 \leq e}F^{e-1}_*{\cred v^\perp} \to \Hom_R(Q_R,R)
\]
as {\cred the} direct limit of $\Psi_n$.
Then the map {\cred $\Psi$} is injective and its image coincides with the set of maps of finite rank.
Let $({\cred F^i_*}g_i)_i\in F_*R \oplus \bigoplus_{2 \leq e}F^{e-1}_*{\cred v^\perp}$. 
Then $\Psi((g_i)_i)$ is of rank $n$ if and only if $g_i=0$ for all $i>n$.
In this case, we denote the image by $\psi_{(g_1,\ldots,g_n)}$.
Therefore, every rank $n$ homomorphism can be denoted by such a form.
\end{defn}

\begin{lem}\label{lem:condition for inducing map}
Let $\psi:=\psi_{(g_1,\ldots,g_n)}$ be a{\cred n} {\cred $R$}-homomorphism of rank $n$, and $I=(f_1,\ldots,f_m)$ and $J$ ideals of $R${\cred , where $f_1, \ldots, f_m \in R$}.
Then the following conditions are equivalent.
\begin{itemize}
    \item[\textup{(1)}] $\psi({\cred F_*W_n(I) Q_{R,n}}) \subseteq J$,
    \item[\textup{(2)}] for all $1 \leq s \leq n$ and $x \in I$,
    \[
    \sum_{r=0}^{n-s} u^r(F^{r}_* {\cred (}g_{r+s}\Delta_{r}(x) {\cred )}) \in J^{[p^s]},
    \]
    and
    \item[\textup{(3)}] for all $1 \leq s \leq n$ and $1 \leq j \leq m$,
    \[
    \sum_{r=0}^{n-s} u^r(F^r_* {\cred (}g_{r+s}\Delta_r(f_j){\cred )}) \in J^{[p^s]}.
    \]
\end{itemize}
\end{lem}

\begin{proof}
First, we prove (1) $\Rightarrow$ (2).
We take an integer $s \in \{1,\ldots,n\}$ and $x \in I$.
For every $p$-monomial $y \in R$, we have
\[
\psi({\cred F_* (}V^{s-1}[xy]{\cred )})=\sum_{r=0}^{n-s} u^{r+s}(F^{r+s}_*{\cred (}g_{r+s}\Delta_{r}(xy){\cred )})
\]
by the definition of $\Psi$ and Lemma \ref{lem:structure of homomorphism for n}.
Since $V^{s-1}[xy]\in W(I)$, the above element is contained in $J$.
Since $y$ is a $p$-monomial, we have $\Delta_{r}(xy)=y^{p^{r}}\Delta_{r}(x)$ by Proposition \ref{prop:rule for delta map} {\cred (1) and (4)}.
Therefore, we have
\[
u^s(F^s_* {\cred (}y\sum_{r=0}^{n-s} u^{r}(F^r_* {\cred (}g_{r+s}\Delta_r(x) {\cred )}) {\cred )}) \in J.
\]
Since $p$-monomials generate $R$ as groups, we obtain that
\[
u^s(F^s_* {\cred (}R\sum_{r=0}^{n-s} u^{r}(F^r_* {\cred (}g_{r+s}\Delta_r(x){\cred )} ){\cred )}) \subseteq J.
\]
Thus, we obtain (2) by the proof of \cite[Lemma 1.6]{Fedder}.

We note that (2) $\Rightarrow$ (3) is clear.
We finish the proof of this lemma {\cred by} proving (3) $\Rightarrow$ (1).
We put the ideal $I'$ of $\var{W}(R)$ generated by
\[
\{V^{s-1}[f_j] \mid 1 \leq j \leq m,\ 1 \leq s \}. 
\]
Then $I'\subseteq W(I){\cred \var{W}(R)}$.
\begin{claim}
$\psi(I')\subseteq J$ holds. 
\end{claim}
\begin{claimproof}
It is enough to show that $\psi({\cred F_* (}\alpha V^{s-1}[f_j] {\cred )}) \in J$ for $1 \leq j \leq m$, $1 \leq s$, and  $\alpha \in \var{W}(R)$.
Since $\psi$ is of rank $n$, we may assume that $s \leq n$.
Furthermore, $\alpha$ is denoted by a sum of {\cred some} elements of the set
\[
\{ V^{t-1}[y] \mid y \colon \text{$p$-monomial},\ 1 \leq t \leq n\}
\]
and an element of $\mathrm{Im}(V^n)$.
Since $\psi$ is of rank $n$, it is enough to show that $\psi({\cred F_* (}V^{s-1}[yf_j] {\cred )})$ {\cred $\in J$}
for $1 \leq j \leq m$, $1 \leq s \leq n$, and a $p$-monomial $y \in R$.
By the argument above and the assumption of (3), we have
\[
\psi({\cred F_* (}V^{s-1}[yf_j] {\cred )})=u^s(F^s_*{\cred (}y\sum_{r=0}^{n-s} u^r(F^r_*{\cred (}g_{r+s}\Delta_r(f_j){\cred )}){\cred )}) \in u^s(F^s_*J^{[p^s]}) \subseteq J, 
\]
as required.
\end{claimproof}\\
We take $\alpha:=(a_0,a_1,\ldots) \in W(I)$.
Since $a_i \in I$ for all $i$, we denote $\alpha$ by a sum of an element of $I'$ and an element of $\mathrm{Im}(V^n)$.
Since $\psi$ is of rank $n$, we have $\psi(\alpha) \subseteq \psi(I')$.
Therefore, we obtain (2) by the above claim.
\end{proof}

\begin{thm}\label{thm:Fedder's type criterion for general ideals}
Let $I$ be an ideal of $R$ and $n$ a positive integer.
Then $R/I$ is $n$-quasi-$F$-split if and only if there exist $g_1,\ldots , g_n \in R$ such that
\begin{itemize}
    \item[\textup{(1)}] $g_1 \notin \m^{[p]}$, $u(F_*g_i)=0$ for $i \geq 2$, and
    \item[\textup{(2)}] for every $s\in\{1,\ldots, n\}$ and $x \in I$,
    \[
    \sum^{n-s}_{r=0}u^r(F^r_*{\cred (}g_{r+s}\Delta_r(x){\cred )}) \in I^{[p^s]}.
    \]
\end{itemize}
\end{thm}

\begin{proof}
First, we assume that $R/I$ is $n$-quasi-$F$-split.
Then the {\cred $R$-module homomorphism} $R/I \to Q_{R/I,n}$ splits.
Since $Q_{R,n}$ is a free $R$-module by Lemma \ref{lem:basis for Witt ring}, the splitting lifts to $Q_{R,n}$, that is, we obtain {\cred an $R$-module homomorphism} $\psi \colon Q_{R,n} \to R$ such that 
\[
\psi({\cred F_*}W_n(I) {\cred Q_{R,n}}) \subseteq I
\]
and $\psi(1)=1$.
By Lemma \ref{lem:structure of homomorphism for n}, there exist ${g_1,\ldots,g_n \cred \in R}$ such that $u(F_*g_i)=0$ for $i \geq 2$ and $\psi=\psi_{(g_1,\ldots,g_n)}$.
By Lemma \ref{lem:condition for inducing map}, $g_1,\ldots,g_n$ satisfy the condition $(2)$ in Theorem \ref{thm:Fedder's type criterion for general ideals}.
We note that $\psi({\cred F_* 1})=u(F_*g_1)$, since $\Delta_r(1)=0$ for $1 \leq r$.
Since $\psi(1)=1 \notin \m$, we have $g_1 \notin \m^{[p]}$.
Therefore, $g_1,\ldots,g_n$ satisfy the desired conditions.

Next, we assume that there exist $g_{1}, \ldots, g_{n}$ satisfying $(1)$ and $(2)$.
We define {\cred an $R$-module homomorphism} $\psi \colon Q_{R,n} \to R$ {\cred by} $\psi:=\psi_{(g_1,\ldots,g_n)}$.
Then $\psi$ induces {\cred an $R/I$-module homomorphism} $Q_{R/I,n} \to R/I$ by the condition (2) and Lemma \ref{lem:condition for inducing map}.
Since $g_1 \notin \m^{[p]}$, there exists $i$ such that ${\cred u_i}(F_*g_1) \notin \m$.
Thus, we have $\psi({\cred F_*( v_d/v_i)})={\cred u_i}(F_*g_1) \notin \m$, and we have the surjectivity of the evaluation map
\[
\Hom_{{\cred R/I}}(Q_{R/I,n},R/I) \to R/I,
\]
which shows $R/I$ is $n$-quasi-$F$-split.
\end{proof}

\begin{lem}\label{lem: g_n change}
Let $g_1,\ldots, g_n\in R$ satisfying $u(F_*g_i)=0$ for $i \in\{2,\ldots,n\}$.
Let $f_1,\ldots,f_m$ be a regular sequence {\cred in $\m$}, $I:=(f_1,\ldots,f_m)$, and $f:=f_1 \cdots f_m$.
We write
\[
G_{s,x}:=\sum_{r=0}^{n-s} u^r(F^r_*{\cred (}g_{r+s}\Delta_r(x){\cred )})
\]
for $x \in R$ and $s\in\{1,\ldots,n\}$.
Then the following are equivalent.
\begin{itemize}
    \item[\textup{(1)}] $G_{s,f_j} \in I^{[p^s]}$ for all $1 \leq s \leq n$ and $1 \leq j \leq m$,
    \item[\textup{(2)}] there exist $h_1,\ldots,h_n$ {\cred $ \in R$} such that $h_i':=f^{p^i-p}h_i \equiv g_i \mod I^{[p^i]}$ for $i \geq 1$, $u(F_*h_i)=0$ for $i \geq 2$,
    \[
    h_s-u(F_*{\cred (}h_{s+1}\Delta_1(f^{p-1}){\cred )}) \in (I^{[p]} \colon I)
    \]
     for all $1 \leq s \leq n-1$, and $h_n \in (I^{[p]} \colon I)$.
\end{itemize}
\end{lem}

\begin{proof}
We note that both (1) and (2) are not affected if we replace a sequence $g_1, \ldots ,g_n$ by $g_1',\ldots,g_n'$ satisfying $g_i \equiv g_i' \mod I^{[p^i]}$ for all $i$.
First, we prove (1)$\Rightarrow$ (2).
By taking $s=n$, we have $g_n \in (I^{[p^n]} \colon I)$, which coincides with $f^{p^n-1}+I^{[p^n]}$ by Lemma \ref{lem:colon rule for regular sequence}.
When $n=1$, it implies (2).
Therefore, we assume that $n$ is at least two.
Thus there exists $h_n$ such that $g_n \equiv f^{p^n-p}h_n \mod I^{[p^n]}$.
Then $0=u(F_*g_n) \equiv f^{p^{n-1}-1}u(F_*h_n) \mod I^{[p^{n-1}]}$.
Therefore, we have
\[
u(F_*h_n) \in (I^{[p^{n-1}]} \colon f^{p^{n-1}-1}) = I
\]
by Lemma \ref{lem:colon rule for regular sequence}.
Replacing $h_n$ by $h_n-u(F_*h_n)^pv$, then $u(F_*h_n)=0$ and
\[
g_n \equiv f^{p^{n}-p}h_n=:h_n'\mod f^{p^{n}-p} I^{[p]}.
\]
Since $f^{p^{n}-p} I^{[p]} \subseteq I^{[p^n]}$, we have $g_n \equiv  h_n' \mod I^{[p^n]}$.
Furthermore, since $h_n' \in (I^{[p^n]} \colon I)$, we have
\[
h_n \in ((I^{[p^n]} \colon I) \colon f^{p^{n}-p})=(I^{[p]} \colon I)
\]
by Lemma \ref{lem:colon rule for regular sequence}.
Therefore, we obtain $h_n$ as desired.
Next, assuming that there exist $h_{s+1},\ldots,h_n$ with the desired condition, let us find $h_s$.
The condition $(i)$ is not affected if we replace $g_i$ by $h_i'$; we may assume that $g_i=h'_i$ for $i>s$.
First, we note that $u^r(F^r_*{\cred (}g_{r+s}\Delta_{r}(f_j){\cred )}) \in I^{[p^s]}$ for $2 \leq r$ and $1 \leq j \leq m$.
Indeed, $g_{r+s} \in f^{p^{r+s}-p}R$ and $\Delta_r(f_j) \in f_j^{p^r-p}R$ by Theorem \ref{thm:delta formula}, we have $u^r(F^r_*{\cred (} g_{r+s}\Delta_{r}(f_j){\cred )}) \in f_j^{p^s}R $ if $r>1$.
Thus, by assumption, we have
\[
g_sf_j+u(F_*{\cred (}g_{s+1}\Delta_1(f_j){\cred )}) \in I^{[p^s]}.
\]
We note that $g_{s+1}/f^{p(p-1)}=f^{p^{s+1}-p^2}h_{s+1}$.
Next, we prove 
\[
u(F_*{\cred (}f^{p^{s+1}-p^2}h_{s+1}\Delta_1(f^{p-1}f_j){\cred )}) \in I^{[p^s]}.
\]
Since $f^{p-1}f_j \in f_j^{p}R$, we have $\Delta_1(f^{p-1}f_j) \in f_j^{p^2}+F(R)$.
Then, since 
\begin{align*}
    &f^{p^{s+1}-p^2}h_{s+1} \in f^{p^{s+1}-p^2}R \subseteq f_j^{p^{s+1}-p^2}R, 
\end{align*}
we have
\[
f^{p^{s+1}-p^2}h_{s+1}\Delta_1(f^{p-1}f_j) \in f_j^{p^{s+1}}+h_{s+1}F(R),
\]
and thus its image by $u$ is contained in $I^{[p^s]}$.
Since 
\[
\Delta_1(f_jf^{p-1})\equiv f_j^p\Delta_1(f^{p-1})+f^{p(p-1)}\Delta_1(f_j) \mod F(R),
\]
we have
\[
g_sf_j-u(F_*{\cred (}f^{p^{s+1}-p^2}h_{s+1}\Delta_1(f^{p-1}){\cred )})f_j \in I^{[p^s]}
\]
for all $j\in \{1,\ldots,m\}$.
Thus, we have
\[
g_s-f^{p^s-p}u(F_*{\cred (}h_{s+1} \Delta_1(f^{p-1}){\cred )}) \in (I^{[p^s]} \colon I)=f^{p^s-1}R+I^{[p^s]},
\]
and $g_s \in f^{p^s-p}+I^{[p^s]}$.
By an argument similar too above, we can find $h_s$ such that $h_s':=f^{p^s-p}h_s \equiv g_s$ mod $I^{[p^s]}$ and $u(F_*h_s)=0$ if $s \geq 2$.
Then we have
\[
f^{p^s-p}(h_s-u(F_* {\cred (}h_{s+1}\Delta_1(f^{p-1}){\cred )})) \in (I^{[p^s]} \colon I),
\]
and thus 
\[
h_s-u(F_* {\cred (}h_{s+1}\Delta_1(f^{p-1}){\cred )}) \in (I^{[p]} \colon I)
\]
by Lemma \ref{lem:colon rule for regular sequence}.
Therefore, we found a sequence $h_1,\ldots,h_n$ with the desired conditions.

Next, we prove the converse direction.
Since $g_i \equiv h'_i \mod I^{[p^i]}$, we may replace $g_i$ with $h'_i$.
In this case, by an argument similar to the above, we have
\[
f^{p^{s+r-1}-p}u(F_*{\cred (} h_{s+r}\Delta_1(f^{p-1}f_j){\cred )}) \in I^{[p^s]}, \textup{ and }
\]
\[
u^r(F^r_*{\cred (}g_{r+s}\Delta_r(f_j){\cred )}) \in I^{[p^s]}.
\]
for all $ s\in\{1,\ldots,,n\}$ and $ 2 \geq r$, 
Therefore, we have
\[
f^{p^s-p}f_j(h_s -u(F_*{\cred (}h_{s+1}\Delta_1(f^{p-1}){\cred )})) \equiv G_{s,f_j} \mod I^{[p^s]}.
\]
Since the left-hand side is contained in $I^{[p^s]}$ by the assumption, we have $G_{s,f_j} \in I^{[p^s]}$, as required.
\end{proof}

\begin{lem}\label{lem:pre-Fedder's criterion}
Let $f_1,\ldots,f_m$ be a regular sequence {\cred in $\m$}, $I:=(f_1,\ldots,f_m)$, $f:=f_1 \cdots f_m$, and $n$ a positive integer.
Suppose that $f^{p-1} \in m^{[p]}$.
Then $R/I$ is $n$-quasi-$F$-split if and only if there exists $h_2,\ldots,h_n$ {\cred $\in R$} such that
\begin{itemize}
    \item[\textup{(1)}] $u(F_*h_i)=0$ for $i \geq 2$,
    \item[\textup{(2)}] $h_s-u(F_*{\cred (}h_{s+1}\Delta_1(f^{p-1}){\cred )}) \in (I^{[p]} \colon I)$ for $2 \leq s \leq n-1$, $h_n \in (I^{[p]} \colon I)$, and
    \item[\textup{(3)}] $u(F_*{\cred (}h_2\Delta_1(f^{p-1}){\cred )}) \notin \m^{[p]}$.
\end{itemize}
\end{lem}

\begin{proof}
First, we assume that $R/I$ is $n$-quasi-$F$-split.
Then there exist $g_1,\ldots,g_n\in R$ satisfying the conditions (1) and (2) in Theorem \ref{thm:Fedder's type criterion for general ideals}.
By Lemma \ref{lem: g_n change}, there exist $h_1, \ldots,h_n$ satisfying the condition (2) in Lemma \ref{lem: g_n change}. 
Thus $h_1,\ldots,h_n$ satisfy the conditions (1) and (2)in this lemma.
Since $h_1 \equiv g_1 \mod I^{[p]}$ and $g_1\notin\m^{[p]}$, it follows that $h_1\notin\m^{[p]}$.
By Lemma \ref{lem: g_n change}, we have
\[
h_1-u(F_*{\cred (}h_2\Delta_1(f^{p-1}){\cred )}) \in (I^{[p]} \colon I)=f^{p-1}R+I^{[p]}.
\]
Since $f^{p-1} \in \m^{[p]}$, we have 
\[
u(F_*{\cred (}h_2\Delta_1(f^{p-1}){\cred )}) \notin \m^{[p]},
\]
and we obtain (3).

Next, we assume the conditions (1)--(3) hold.
We define $h_1$ as
\[
h_1:=u(F_*{\cred (}h_2\Delta_1(f^{p-1}){\cred )}).
\]
Then $h_1,\ldots,h_n$ satisfy the condition (2) in Lemma \ref{lem: g_n change}.

We set $g_i:=f^{p^i-p}h_i$. Then $g_1,\ldots,g_n$ satisfy the condition (1) in Lemma \ref{lem: g_n change}.
Therefore, they satisfy both conditions in Theorem \ref{thm:Fedder's type criterion for general ideals}.
Thus, we have $R/I$ is $n$-quasi-$F$-split.
\end{proof}

\begin{thm}[A Fedder type criterion for quasi-$F$-splitting]\label{thm:Fedder's criterion}
We use the notation in Convention \ref{conv:regular local ring}.
Let $f_1,\ldots,f_m$ be a regular sequence in $\m$, $I:=(f_1,\ldots,f_m)$, and $f:=f_1 \cdots f_m$.
We define an $R$-module homomorphism $\theta$ by
\[
\theta \colon{\cred v^\perp} \to R \ ;\ F_*a \mapsto u(F_*{\cred (}\Delta_1(f^{p-1})a{\cred )}).
\]
We define a sequence of ideal $\{I_s\}_{s>0}$ by $I_1:=(I^{[p]} \colon I)$ and 
\[
I_{s+1}:=\theta(F_*I_s \cap v^{\perp})+I_1.
\]
Then we have
\[
\sht(R/I)=\mathrm{inf}\{n \mid I_n \nsubseteq \m^{[p]} \}.
\]
\end{thm}

\begin{proof}
By Fedder's criterion (\cite{Fedder}), the quasi-$F$-split height is one if and only if $f^{p-1} \notin \m^{[p]}$, and they are also equivalent to saying $I_1 \nsubseteq I^{[p]}$ by Lemma \ref{lem:colon rule for regular sequence}.
Thus, we may assume that $f^{p-1} \in \m^{[p]}$, or equivalently that $\sht(R/I) \geq 2$.
First, we assume that $I_n \nsubseteq \m^{[p]}$.
Then there exists a sequence $h_2,\ldots,h_n$ such that $u(F_*h_i)=0$, $h_s-\theta(F_*h_{s+1}) \in (I^{[p]} \colon I)$, $h_n \in I_1$,
for $2 \leq s \leq n-1$, and $\theta(F_*h_2) \notin \m^{[p]}$.
Therefore, the sequence $h_2,\ldots,h_n$ satisfies the conditions (1), (2), and (3) in Lemma \ref{lem:pre-Fedder's criterion}, and thus $\sht(R/I)\leq n$.

Next, we assume that $\sht(R/I)\leq n$.
Then we can find $h_2,\ldots,h_n$ satisfying the conditions (1), (2), and (3) in Lemma \ref{lem:pre-Fedder's criterion}.
We prove that $h_{n-s+1} \in I_s$ for $1 \leq s \leq n-1$ by induction on $s$.
If $s=1$, then $F_*h_n \in F_*I_1 \cap v^\perp$.
We assume $h_{n-i+1} \in I_i$ for $i<s \leq n-1$.
By the condition (1) in Lemma \ref{lem:pre-Fedder's criterion} and $n-s+2 \geq 2$, we have $u(F_*h_{n-s+2})=0$, and thus $F_*h_{n-s+2} \in F_*I_{s-1} \cap v^{\perp}$.
By the condition (2) in Lemma \ref{lem:pre-Fedder's criterion}, we have
\[
h_{n-s+1}-\theta(F_*h_{n-s+2}) \in (I^{[p]} \colon I),
\]
and hence we have $h_{n-s+1} \in I_s$.
Thus, we obtain $F_*h_2 \in F_*I_{n-1} \cap v^\perp$, and $\theta(F_*h_2) \in I_n$.
By the condition (3) in Lemma \ref{lem:pre-Fedder's criterion}, we have $\theta(F_*h_2) \notin \m^{[p]}$.
Therefore, we obtain $I_n \nsubseteq \m^{[p]}$.
\end{proof}

\begin{rmk}\label{rmk:increasing remark}
The sequence of ideals $\{I_n\}_n$ in Theorem \ref{thm:Fedder's criterion} is an increasing sequence.
Indeed, we have $I_1 \subseteq I_2$ by definition, and
if $I_{n-1} \subseteq I_n$, then
\[
I_n=\theta(F_{*}I_{n-1} \cap v^\perp)+I_1 \subseteq \theta(F_{*}I_n \cap v^\perp)+I_1 =I_{n+1}.
\]
\end{rmk}

\begin{eg}\label{eg:easy example}
Let $k$ be a perfect field of characteristic {\cred $p=2$}.
Let $R:=k[[x,y,z]]$ and
\[
f:=x^3+y^3+z^3.
\]
We prove $\sht(R/f)=2$.
We take a basis of $F_*R$ over $R$ as
\[
\{ F_*{\cred (}x^iy^jz^l{\cred )} \mid 0 \leq i, j, l \leq p-1 \}.
\]
{\cred We apply Theorem \ref{thm:Fedder's criterion} for $m=1$ and $f_1 =f$. Then $I = (f).$}
We define $\Delta_1$, $\theta$ and $\{I_n\}_n$ as in Theorem \ref{thm:Fedder's criterion} via this basis.
Since $f^{p-1}=f \in \m^{[2]}$, it follows that $I_1$ {\cred $=((f)^2:(f)) =(f)$}
$\subseteq \m^{[2]}$, and thus $\sht(R/f) \geq 2$.
We next compute $\Delta_1(f)$.
Since the decomposition $f=x^3+y^3+z^3$ is a $p$-monomial decomposition, 
\[
x^3y^3+y^3z^3+x^3z^3.
\]
{\cred is a representative of $\Delta_1 (f)$}
by Remark \ref{eg:explicit delta}.
By abuse of notation, we also denote this representative by $\Delta_1(f)$.
Therefore, we have
\[
f \Delta_1(f) \equiv x^3y^3z^3 \mod \m^{[4]}.
\]
As $u(F_*f)=0$ and
\[
\theta(F_*f)=u(F_*{\cred (}f\Delta_1(f){\cred )}) \equiv xyz \mod \m^{[2]},
\]
we have $\theta(F_*f)\in$
{\cred $\theta(F_* I_1 \cap v^{\perp})\subseteq$}
$I_2$ and $\theta(F_*f) \notin \m^{[2]}$, and in particular, $I_2 \nsubseteq \m^{[p]}$.
By Theorem \ref{thm:Fedder's criterion}, we have $\sht(R/f)=2$.
\end{eg}

\begin{eg}
Let $k$ be a perfect field of characteristic {\cred $p=2$}.
Let $R:=k[[x,y,z]]$ and
\[
f:=x^2+y^3+z^5.
\]
We prove $\sht(R/f)=4$.
We take a basis of $F_*R$ over $R$ as in Example \ref{eg:easy example} and define $\Delta_1$, $\theta$ and $\{I_n\}_n$ as in Theorem \ref{thm:Fedder's criterion} via this basis.
First, we {\cred prove} that $\sht(R/f) \leq 4$.
It is enough to show that $I_4 \nsubseteq \m^{[p]}$.
By {\cred Remark \ref{eg:explicit delta},} an element
\[
x^2y^3+x^2z^5+y^3z^5
\]
is a representative of $\Delta_1(f)$.
We can define $\theta \colon F_*R \to R$ by using $\Delta_1(f)$.
Next, we have
\[
f\Delta_1(f)=
{\cred x^4y^3 + x^4z^5 + x^2y^3z^5+x^2y^6+y^6z^5+x^2z^{10}+y^3z^{10}}
\]

Thus, we have the basis expansion
\[
\theta=xyv_{xz}^*+xz^2v^*_{xy}+yz^2v^*_{x}
\]
and 
\[
\theta(F_*f \cdot \underline{\ \ })=
{\cred (x^2y+yz^5) v_{xz}^{\ast} + (x^2z^2+y^3z^2)v_{xy}^*+xyz^2v_x^*+(xy^3+xz^5)v_{xyz}^*},
\]
where $v^*_{xy}$ is the dual of $F_*xy$ with respect to the basis and the others are defined in a similar way.
We put $g_1:=xf \in I_1$. Then 
\[
u(F_*g_1)=v^*_{yz}(F_*f)=0\ \text{and}\  \theta(F_*g_1)=\theta(F_*{\cred (}xf{\cred )})=xyz^2,
\]
and $g_2:=xyz^2 \in I_2$.
Since $u(F_*g_2)=0$ and $\theta(F_*g_2)=xz^3$, 
we have $g_3:=xz^3 \in I_3$.
Moreover, since $u(F_*g_3)=0$ and $\theta(F_*g_3)=xyz$,
we have $xyz \in I_4$.
We then deduce from $xyz\notin \m^{[2]}$ that $I_4$ is not contained in $\m^{[2]}$.
Therefore, we have $\sht(R/f) \leq 4$.

Next, we prove that $\sht(R/f)=4$.
It is enough to show that $I_3\subset\m^{[2]}$.
We use the representative $x^2y^3+x^2z^5+y^3z^5$ of $\Delta_1(f)$.
By the definition of $I_3$, we have $I_3 \subseteq \theta^2(F^2_*fR)+\theta(F_*fR)+fR$.
We note $f \in \m^{[2]}$.
We also have
\[
\theta(F_*fR)=(x^2y+xy^2, x^2z^2+y^3z^2, xyz^2,xz^5) \subseteq \m^{[2]}
\]
by the basis expansion of $\theta(F_*f \cdot \underline{\ \ })$, and obtain $\theta^2(F^2_*{\cred (}fR{\cred )})\subseteq\m^{[2]}$ by the basis expansion of $\theta$.
Therefore, we conclude that $I_3\subset\m^{[2]}$. 
\end{eg}

\begin{cor}\label{cor:complete intersection versus hypersurface}
{\cred We use the notation in Convention \ref{conv:regular local ring}.}
Let $f_1,\ldots,f_m$ be a regular sequence {\cred in $\m$}, $I\coloneqq(f_1,\ldots,f_m)$, and $f\coloneqq f_1 \cdots f_m$.
Then $\sht({\cred R}/I) \leq \sht({\cred R}/f)$.
\end{cor}

\begin{proof}
We define an $R$-module homomorphism $\theta$ by
\[
\theta \colon{\cred v^\perp} \to R \ ;\ F_*a \mapsto u(F_*{\cred (}\Delta_1(f^{p-1})a{\cred )}).
\]
We define a sequence of ideals $\{I_n\}_n$ and $\{I'_n\}_n$ by
\begin{align*}
    &I_1:=(I^{[p]} \colon I),\ I'_1:=(f^{p-1}) \\
    & I_{n+1}:=\theta(F_*I_n \cap v^\perp)+I_1 \\
    & I'_{n+1}:=\theta(F_*I'_n \cap v^\perp)+I'_1,
\end{align*}
inductively.
Then $I'_n\subseteq I_n$ for all $n$ by definition, and we obtain the assertion by Theorem \ref{thm:Fedder's criterion}.
\end{proof}

\begin{rmk}
By Corollary \ref{cor:complete intersection versus hypersurface}, if $\sht({\cred R}/f)=1$, then so is $\sht({\cred R}/I)$.
Moreover, the converse also holds since $I^{[p]} \subseteq \m^{[p]}$ (cf.~\cite[Proposition 2.1]{Fedder}).
On the other hand, the equality $\sht({\cred R}/I) =\sht({\cred R}/f)$ does not hold in general (see Example \ref{eg:height infty}).
However, we will see that the equality holds for `Calabi-Yau varieties' in Theorem \ref{thm:Fedder's criterion for Calabi-Yau}. 
\end{rmk}

\begin{cor}\label{cor:Fedder's criterion for quasi-F-splitting}
{\cred We use the notation in Convention \ref{conv:regular local ring}.}
Let $f_1,\ldots,f_m$ be a regular sequence {\cred in $\m$}, $I:=(f_1,\ldots,f_m)$, and $f:=f_1 \cdots f_m$.
We define an $R$-module homomorphism $\theta$ by
\[
\theta \colon {\cred v^\perp} \to R\ ;\ F_*a \mapsto u(F_*{\cred (}\Delta_1(f^{p-1}){\cred a}{\cred )}).
\]
Then the following hold.
\begin{itemize}
    \item[\textup{(1)}] There exists the {\cred smallest} ideal $I_{\infty}$ satisfying
    \[
    I_{\infty} \supseteq \theta(F_{*}I_{\infty} \cap v_d^\perp )+(I^{[p]} \colon I).
    \]
    \item[\textup{(2)}] $R/I$ is quasi-$F$-split if and only if $I_{\infty} \nsubseteq \m^{[p]}$.
\end{itemize}
\end{cor}

\begin{proof}
By Remark \ref{rmk:increasing remark} and the Noetherian property of $R$, there exists an ideal $I_{\infty}$ such that $I_{\infty}=I_n$ for large enough $n$.
Therefore, $I_{\infty}$ satisfies 
\[
I_{\infty} = \theta(F_{*}I_{\infty} \cap v^\perp )+(I^{[p]} \colon I).
\]
Next, we take an ideal $J\subseteq R$ satisfying 
\[
J \supseteq \theta(F_{*}J \cap v^\perp )+(I^{[p]} \colon I).
\]
Then $J$ contains $(I^{[p]}\colon I)=I_1$.
If $I_s$ is contained in $J$, then we have
\[
I_{s+1}=\theta(F_*I_s \cap v^\perp)+I_1 \subseteq \theta(F_*J \cap v^\perp)+I_1 \subseteq J.
\]
Therefore, we have $I_n \subseteq J$ for all $n$ by induction, and in particular, $I_{\infty}$ is contained in $J$.
Thus, we obtain (1).

By the construction of $I_{\infty}$ and Theorem \ref{thm:Fedder's criterion}, we obtain (2).
\end{proof}

\begin{eg}\label{eg:height infty}
Let $k$ be a perfect field of characteristic {\cred $p=2$}.
Let $R:=k[[x,y,z,w]]$ and $g:=x^3+y^3+z^3+xyzw^2$.
First, we show that $R/g$ is not quasi-$F$-split.
We take a basis of $F_*R$ over $R$ as in Example \ref{eg:easy example}, and define $\Delta_1$, $\theta$, and $\{I_n\}_n$ as in Theorem \ref{thm:Fedder's criterion} via this basis, {\cred then we have
\[
\Delta_1(g)=x^3y^3+x^3z^3+y^3z^3+xyzw^2(x^3+y^3+z^3).
\]
We define an ideal $J$ by
\[
J:=(xy,yz,xz,x^4w,y^4w,z^4w) \cap (x^2,y^2,z^2)+(g),
\]
then we have
\[
J \supseteq \theta(F_*J \cap v^\perp)+(g).
\]
We give a proof of the inclusion.
Take \( a \in J \) such that \( u(F_*a) = 0 \).  
Then there exist \( b_1, b_2, b_3, c \in R \) such that
\[
a = b_1x^2 + b_2y^2 + b_3z^2 + cg
\]
and
\[
b_1x^2 + b_2 y^2 + b_3 z^2 \in J_1 := (xy, yz, xz, x^4w, y^4w, z^4w).
\]
Since
\[
b_1x^2  \in J_1 + (y, z) \subseteq (x^4w, y, z)
\]
and \( x, y, z \) is a regular sequence in $\m$, it follows that \( b_1 \in (x^2, y, z) \).  
By the same argument, we obtain \( b_2 \in (x, y^2, z) \) and \( b_3 \in (x, y, z^2) \).  
In particular, we have
\[
a - cg \in J' := (x^4, x^2y, x^2z, xy^2, y^4, y^2z, xz^2, yz^2, z^4).
\]
Furthermore, by direct computation, we can easily show that \( J'\Delta_1(g) \subseteq J^{[2]} \).
Therefore, we obtain
\[
\theta(F_*(a - cg)) \in u(F_*(J'\Delta_1(g))) \subseteq u(F_*J^{[2]}) \subseteq J.
\]
Therefore, it remains to show that \( cg\Delta_1(g) \in J^{[2]} \).  
We compute
\[
u(F_*a) = u(F_*b_1)x + u(F_*b_2)y + u(F_*b_3)z + u(F_*(cg)) = 0,
\]
which implies \( u(F_*(cg)) \equiv w \cdot u(F_*(cxyz)) \in (x, y, z) \).  
Since \( x, y, z, w \) form a regular sequence in $\m$, we deduce that \( c \in (x, y, z) \).  
Moreover, as
\[
g\Delta_1(g) \equiv x^3y^3z^3 \mod J^{[2]},
\]
we conclude that \( cg\Delta_1(g) \in J^{[2]} \), as desired.}
Therefore, by Corollary \ref{cor:Fedder's criterion for quasi-F-splitting}, we have
$I_\infty \subseteq J \subseteq \m^{[2]}$,
and $R/g$ is not quasi-$F$-split.

Next, we confirm that $R/(gw)$ is not quasi-$F$-split.
We denote an $R$-module homomorphism ${\cred v^\perp} \to R\ ;\ F_*a \mapsto u(F_*{\cred (} \Delta_1(gw)a{\cred )})$ by $\theta'$.
Since we have $\Delta_1(gw)=w^2\Delta_1(g)$ {\cred by Proposition \ref{prop:rule for delta map} (1) and (4)}, we obtain
\[
\theta'(F_*a)=w\theta(F_*a)
\]
for all $a \in {\cred v^\perp}$.
Therefore, we have
\[
 \theta'(F_*J \cap v^\perp)+(g) \subseteq \theta(F_*J \cap v^\perp)+(g) \subseteq J, 
\]
and $R/(gw)$ is not quasi-$F$-split.

Moreover, since $R/(g,w) \cong k[[x,y,x]]/(x^3+y^3+z^3)$, it follows that $\sht(R/(g,w))=2$.
Therefore, in Corollary \ref{cor:complete intersection versus hypersurface}, $\sht(R/I)$ can be strictly less than $\sht(R/f)$.
\end{eg}

\begin{cor}\label{cor:non-q-split criterion}
{\cred We use the notation in Convention \ref{conv:regular local ring}.}
Let $f_1,\ldots,f_m$ be a regular sequence {\cred in $\m$}, $I\coloneqq(f_1,\ldots,f_m)$, and $f\coloneqq f_1 \cdots f_m$.
Then the following hold.
\begin{itemize}
    \item[\textup{(1)}] If $f^{p-2} \in \m^{[p]}$, then $\sht(R/I)=\infty$.
    \item[\textup{(2)}] If ${\cred (f^{p-2},I^{[p]})f^{p(p-2)}\widetilde{\Delta_{1}(f)} \subseteq \m^{[p^2]}}$ for some representative ${\cred \widetilde{\Delta_1(f)} \in R}$ 
    {\cred of $\Delta_1(f) \in R/F(R)$} 
    and $f^{p-1} \in \m^{[p]}$, then $\sht(R/I)=\infty$. 
\end{itemize}
\end{cor}

\begin{proof}
In both cases, we have $f^{p-1}\in \m^{[p]}$.
We note that $I_n \subseteq (f^{p-2})+I_1$ for all $n$.
Indeed, since $\Delta_1(f^{p-1}) \equiv -f^{p(p-2)} \Delta_1(f) \mod F(R)$, we have
\[
\theta({\cred v^\perp}) \subseteq (f^{p-2}).
\]
Then (1) immediately follows from this fact and Theorem \ref{thm:Fedder's criterion}.

Next, we prove that $I_n \subseteq \m^{[p]}$ for all $n$ under the assumption in (2).
For $n=1$, we have
\[
I_1:=(I^{[p]} \colon I)=I^{[p]}+(f^{p-1}) \subseteq \m^{[p]}.
\]
For $n \geq 2$, we have
\[
I_{n} := \theta(F_*I_{n-1} \cap v^\perp) +I_1 \subseteq u(F_*{\cred (}\Delta_1(f) f^{p(p-2)}(f^{p-2},I^{[p]}))) + I_1 \subseteq \m^{[p]}+I_1\subseteq \m^{[p]}.
\]
Therefore, $\sht(R/I)=\infty$ by Theorem \ref{thm:Fedder's criterion}.
\end{proof}

\section{Fedder type criteria in the graded case}
In this section, we generalize Theorem \ref{thm:Fedder's criterion} to the case of a complete intersection in a product of (weighted) projective spaces.
In particular, we obtain a simple version of Theorem \ref{thm:Fedder's criterion} in the Calabi-Yau case, i.e., the case where a given complete intersection has a trivial canonical sheaf (see Theorem \ref{thm:Fedder's criterion for Calabi-Yau}).

\begin{conv}\label{conv:graded ring}
In this subsection, $k$ is an $F$-finite field of characteristic $p>0$ and $S:=k[x_1,\ldots,x_N]$ is a polynomial ring with a $\mathbb{Z}_{\geq 0}^m$-graded structure.
We assume that the degree of $x_i$ is non-zero for all $i$, which is denoted by $\mu_i$
We define $\mu$ by $\mu:=\sum \mu_i$.
Let 
\[
\m:=\bigoplus_{h \in \mathbb{Z}_{\geq 0}^m \backslash \{0\}} S_h =(x_1,\ldots,x_N).
\]
Let $R:=S_\m$. Then $(R,\m,k)$ is a regular local ring, where we denote the extension of $\m$ in $R$ by $\m$ by abuse of notation.
Let $y_1,\ldots,y_s$ be a $p$-basis of $k$. 
{\cred We denote the elements of 
\[
\{y_1^{j_1}\cdots y_s^{j_s}x_1^{i_1} \cdots x_N^{i_N} \mid 0 \leq j_1,\ldots,j_s, i_1,\ldots,i_N \leq p-1\}
\]
by $v_1, \ldots, v_{d}$.
Then $F_{*}v_1, \ldots, F_*v_{d}$
}
is a basis of $F_*S$ over $S$ as in Lemma \ref{lem:standard basis}.
{\cred We may assume that
\[
v:=v_d= {\cred (y_1\cdots y_s x_1 \cdots x_N)^{p-1}}.
\]
}
The dual basis is denoted by $u_1,\ldots,u_d$ and $u:=u_d$.
We note that $\deg(v)=(p-1)\mu$.
Using this basis, we define the notion of $p$-monomials as in Definition \ref{definition:monomial} and ${\cred \Delta_1'} \colon S \to S/F(S)$ as in Definition \ref{definition:delta map}.
Therefore, we have the following commutative diagram:
\[
\xymatrix{
S \ar[r]^-{{\cred \Delta_1'}} \ar[d] \ar@{}[rd]|{\circlearrowleft} & S/F(S) \ar[d] \\
R \ar[r]_-{\Delta_1} & R/F(R).
}
\]
Indeed, if we take a $p$-monomial decomposition $a=\sum a^p_iv_i$, then its image gives a $p$-monomial decomposition in $R$.
For this reason, we denote ${\cred{\Delta}_1'}$ by $\Delta_1$ by abuse of notation.
\end{conv}

\begin{rmk}\label{rmk:same properties}
We can see that $\Delta_1$ for $S$ satisfies the properties as in Proposition \ref{prop:rule for delta map}, \ref{prop:delta gives a splitting}, and Theorem \ref{thm:delta formula}.
For the proof, it is enough to check the injectivity of the map
\[
S/F(S) \to R/F(R).
\]
We prove the injectivity. If the image of $x \in S$ is contained in $F(R)$, then there exists $f,y \in S$ such that $x=(y/f)^p$.
Since $S \to F_*S$ is a split injection, so is the map
\[
F \colon S/f \to F_*(S/f^p).
\]
Now it follows from $y^p=xf^p \in (f^p)$ that $y \in (f)$, and thus $x \in F(S)$.
\end{rmk}

\begin{lem}\label{lem:homogeneous choice}
If $a \in S$ is a homogeneous element of degree $e$, then there exists a representative of $\Delta_1(a)$ which is homogeneous of degree $pe$.
\end{lem}

\begin{proof}
We take the basis expansion $a=\sum^{d}_{i=1} a^p_iv_i$.
Since every $v_i$ is homogeneous, it follows that $a_i^p v_i$ is homogeneous of degree $e$.
By the construction of $\Delta_1$ and Remark \ref{eg:explicit delta}, the element
\[
\sum_{\substack{0 \leq \alpha_1, \ldots,\alpha_d \leq p-1 \\ \alpha_1+\cdots+\alpha_d=p}} \frac{1}{p} \binom{p}{\alpha_1, \ldots ,\alpha_d}(a^p_1v_1)^{\alpha_1} \cdots (a_d^pv_d)^{\alpha_d}
\]
is a representative of $\Delta_1(a)$ which is homogeneous of degree $pe$.
\end{proof}

\begin{cor}\label{cor:graded Fedder}
Let $f_1,\ldots,f_m$ be a homogeneous regular sequence in $S$, $I:=(f_1,\ldots,f_m)$, and $f:=f_1 \cdots f_m$.
We define an $S$-module homomorphism $\theta$ by
\[
\theta \colon {\cred v^\perp} \to S\ ;\ F_*a \mapsto u(F_*{\cred (}\Delta_1(f^{p-1})a{\cred )}).
\]
We define $I_1:=(I^{[p]} \colon I)$, and
\[
I_{s+1}:=\theta(F_*I_s \cap v^\perp)+I_1
\]
inductively.
Then 
\[
\sht(S/I)=\mathrm{min}\{n \mid I_n \nsubseteq \m^{[p]} \}.
\]
\end{cor}

\begin{proof}
The assertion follows from Theorem \ref{thm:Fedder's criterion} and Proposition \ref{prop:multigraded ring versus local ring}.
\end{proof}

\begin{rmk}\label{rmk:homogeneous ideals}
In the setting of Corollary \ref{cor:graded Fedder}, we can see that ideals $I_n$ are homogeneous by induction on $n$ as follows:

First, $I_1=(I^{[p]}\colon I)=I^{[p]}+(f^{p-1})$ is a homogeneous ideal.
We assume that $I_n$ is a homogeneous ideal.
For an element $F_*a \in F_*I_n \cap v^\perp$, we take a homogeneous decomposition $a=\sum a_h$.
Since $u(F_*a_h)$ is homogeneous and $\mathrm{deg}(a_h) \neq \mathrm{deg}(a_g)$ if $h \neq g$, we have $u(F_*a_h)=0$ for all $h$.
Thus, $F_*I_n \cap v^\perp$ is generated by homogeneous elements.
Since $\Delta_1(f^{p-1})$ has a representative which is homogeneous, we have $\theta(F_*I_n \cap v^\perp)$ is also generated by homogeneous elements.
Therefore, $I_{n+1}=\theta(F_*I_n \cap v^\perp)+I_1$ is a homogeneous ideal, as desired.
\end{rmk}

\begin{lem}\label{lem:homogeneous of special degree}
Let $a \in S$ be a homogeneous element of degree $(p^l-1)\mu$.
Then we have
\[
a \equiv b(x_1 \cdots x_N)^{p^l-1} \mod \m^{[p^l]}
\]
for some $b \in k$.
In particular, $a \in \m^{[p^l]}$ if and only if $u^l(F^l_*{\cred (}ac {\cred )})=0$ for every $c \in k$.
\end{lem}

\begin{proof}
We note that
\[
\{F^{l-1}_* {\cred (} \prod_{1 \leq i \leq l} v_{m_i}^{p^{i-1}} {\cred )} \mid 1 \leq m_i \leq d \}
\]
is a basis of $F^l_*S$ over $S$.
This basis is denoted by ${\cred F^l_* w_1,\ldots, F^l_*w_e}$.
We take the basis expansion $a=\sum a_i^{p^l}w_i$.
Since $a$ and $w_i$ are homogeneous, $a_i^{p^l}w_i$ is homogeneous of degree $(p^l-1)\mu$.
Then we have
\[
(p^l-1)\mu=p^l \deg(a_i)+\deg(w_i).
\]
By the choice of $w_i$, the degree of $w_i$ is at most $(p^l-1)\mu$.
Therefore, if $\deg (w_i) \neq (p^l-1)\mu$, then $\deg(a_i) \neq 0$, and in particular, $a_i^{p^l}w_i\in \m^{[p^l]}$.
Furthermore, if $\deg(w_i)=(p^l-1)\mu$, then the degree of $a_i$ is zero and 
\[
w_i= {\cred b_i (x_1 \cdots x_N)^{p^l-1} }
\]
for some $b_i \in k$.
In particular, we have
\[
a \equiv b(x_1\cdots x_N)^{p^l-1} \mod \m^{[p^l]}
\]
for some $b \in k$, and thus the first assertion holds.

Next, we prove that $a \in \m^{[p^l]}$ if and only if $u^l(F^l_*{\cred (}ac{\cred )})=0$ for every $c \in k$.
Suppose that $a\in \m^{[p^l]}$.
Then, for every $c \in k$, we have
$u^l(F^l_* {\cred (}ac {\cred )}) \in \m$. Since $\deg(u^l(F^l_*{\cred (}ac {\cred )}))=0$, we have $u^l(F^l_* {\cred (} ac {\cred )})=0$, as desired.
Suppose that $u^l(F^l_*{\cred (}ac {\cred )})=0$ for all $c \in k$.
Note that
\begin{align*}
    b=u^l(F^l_*{\cred (}b^{p^l}(y_1\cdots y_sx_1 \cdots x_N)^{p^l-1} {\cred )}).
\end{align*}
We put $c:=(by_1 \cdots y_s)^{p^l-1} \in k$.
Since $a \equiv b(x_1 \cdots x_N)^{p^l-1} \mod \m^{[p^l]}$, 
we obtain
\begin{align*}
    0
    &=u^l(F^l_*{\cred (}ac {\cred )}) \\
    &\equiv u^l(F^l_* {\cred (}b^{p^l}(y_1\cdots y_sx_1 \cdots x_N)^{p^l-1} {\cred )}) \mod \m \\
    &=b,
\end{align*}
and in particular, $b\in k\cap\m=0$, as desired.
\end{proof}

\begin{lem}\label{lem:homogeneous of top degree}
Let $a \in S$ be a homogeneous element of degree $(p^l-1)\mu$.
Then the following are equivalent.
\begin{itemize}
    \item[\textup{(1)}] $a \in \m^{[p^l]}$,
    \item[\textup{(2)}] $u^{l-1}(F^{l-1}_* {\cred (}aS {\cred )}) \subseteq \m^{[p]}$,
    \item[\textup{(3)}] $u^{l-1}(F^{l-1}_* {\cred (}ak {\cred )}) \subseteq \m^{[p]}$, and
    \item[\textup{(4)}] $u^l(F^l_* {\cred (}ak {\cred )})=0$.
\end{itemize}
\end{lem}

\begin{proof}
(1) $\Rightarrow$ (2) $\Rightarrow$ (3) is clear.
We prove (3) $\Rightarrow$ (4).
We take $c \in k$.
We put $a':=u^{l-1}(F^{l-1}_* {\cred (}ac {\cred )})$. Then the degree of $a'$ is $(p-1)\mu$.
Since $a'\in \m^{[p]}$, we have $u^l(F^l_* {\cred (} ac {\cred )})=u(F_*a')=0$ by Lemma \ref{lem:homogeneous of special degree}.
Finally, (4) $\Rightarrow$ (1) follows from Lemma \ref{lem:homogeneous of special degree}.
\end{proof}

\begin{thm}[Theorem \ref{Intro:thm:Fedder's criterion for Calabi-Yau}]\label{thm:Fedder's criterion for Calabi-Yau}
Let $f'_1,\ldots,f'_m$ be a homogeneous regular sequence {\cred in $S$}, $I:=(f'_1,\ldots,f'_m)$, and $f:=f'_1 \cdots f'_m$.
Suppose that $f$ is a homogeneous element of degree $\mu$.
We take a representative of $\Delta_1(f^{p-1})$ which is homogeneous of degree $p(p-1)\mu$. We also denote the representative by $\Delta_1(f^{p-1})$.
We define an $S$-module homomorphism $\theta$ by
\[
\theta \colon F_*S \to S\ ;\ F_*a \mapsto u(F_*{\cred (}\Delta_1(f^{p-1})a{\cred )}).
\]
We define $f_n$ for $n\in\Z_{>0}$ by $f_1:=f^{p-1}$ and
$f_n:=f^{p-1}\Delta_1(f^{p-1})^{1+p+\cdots+p^{n-2}}$
{\cred for $n \geq 2.$}
Then
\[
\sht(S/I)=\mathrm{inf}\{n \mid \theta^{n-1}(F^{n-1}_* {\cred (}f^{p-1}k {\cred )}) \nsubseteq \m^{[p]} \}=\mathrm{inf}\{n \mid f_n \notin \m^{[p^n]} \}
\]
holds.
In particular, we have $\sht(S/I)=\sht(S/f)$.
\end{thm}

\begin{proof}
First, we note that if $a$ is a homogeneous element of degree $(p-1)\mu$, then
$\deg(\Delta_1(f^{p-1})a)=(p^2-1)\mu$, and thus 
\[
\deg(\theta(F_*a))=((p^2-1)\mu-(p-1)\mu)/p=(p-1)\mu.
\]
Moreover, we note that the degree of $f_l$ is $(p^l-1)\mu$ and $\theta^{l-1}(F^{l-1}_*f^{p-1})=u^{l-1}(F^{l-1}_*f_l)$.
By Lemma \ref{lem:homogeneous of top degree}, $\theta^{l-1}(F^{l-1}_* {\cred (}f^{p-1}k {\cred )}) \subseteq \m^{[p]}$ if and only if $f_l \in \m^{[p^l]}$, and in particular, 
$\mathrm{inf}\{n \mid \theta^{n-1}(F^{n-1}_* {\cred (} f^{p-1}k {\cred )}) \nsubseteq \m^{[p]} \}=\mathrm{inf}\{n \mid f_n \notin \m^{[p^n]} {\cred \}}$ holds.

Next, we put $n:=\sht(S/I)$.
We may assume that $n \geq 2$.
By Corollary \ref{cor:graded Fedder}, $I_n$ is not contained in $\m^{[p]}$ and $I_{n-1}$ is contained in $\m^{[p]}$, where the sequence of ideals $\{I_n\}$ is as in Corollary \ref{cor:graded Fedder}.
\begin{claim}\label{claim:inclusion}
$u^{l-1}(F^{l-1}_*{\cred (}f_lk {\cred )}) \subseteq I_l$ for all $l \leq n$.
\end{claim}
\begin{claimproof}
We prove the assertion by induction on $l$.
When $l=1$, we have \[f_1=f^{p-1} \in I_1=(f^{p-1})+I^{[p]}\] and the assertion holds.
For $l \geq 2$, we assume that $u^{l-2}(F^{l-2}_*{\cred (}f_{l-1}k {\cred )}) \in I_{l-1}$.
Since $l \leq n$, we have 
\[
u^{l-2}(F^{l-2}_*{\cred (} f_{l-1}k {\cred )}) \subseteq I_{l-1} \subseteq \m^{[p]}.
\]
By Lemma \ref{lem:homogeneous of top degree} (3)$\Rightarrow$(4), we have $u^{l-1}(F^{l-1}_* {\cred (}f_{l-1}k {\cred )})=0$, and thus
$F_* {\cred (}u^{l-2}(F^{l-2}_* {\cred (} f_{l-1}k {\cred )}) {\cred )} \subseteq F_*I_{l-1} \cap v^\perp$.
Therefore, we conclude that ${\cred \theta(F_* (u^{l-2}(F^{l-2}_* {\cred (} f_{l-1}k {\cred )})))}=u^{l-1}(F^{l-1}_* {\cred (}f_lk {\cred )}) \subseteq I_l$.
\end{claimproof} \\
By Claim \ref{claim:inclusion}, we have
\[
u^{l-1}(F^{l-1}_* {\cred (}f_lk {\cred )}) \subseteq I_l \subseteq \m^{[p]}
\]
for $l \leq n-1$.
Therefore, we have
\[
\sht(S/I)=n \leq \mathrm{inf}\{l \mid \theta^{l-1}(F^{l-1}_* {\cred (} f^{p-1}k {\cred )}) \nsubseteq \m^{[p]}\}.
\]
For the proof of the converse inequality, we prove the following claim.
\begin{claim}\label{claim:ip claim}
$\theta(F_*I^{[p]} \cap v^\perp) \subseteq I_1$ holds.
\end{claim}

\begin{claimproof}
We note that $F_*I^{[p]} \cap v^{\perp}$ is generated by elements of the form $x{f'_j}^p$ with $u(F_*x)=0$.
By Proposition \ref{prop:rule for delta map} (4) and Remark \ref{rmk:same properties}, we have
\[
\Delta_1({f'_j}^pf^{p-1}) \equiv {f_j'}^p\Delta_1(f^{p-1})+f^{p(p-1)}\Delta_1(f'_j) \mod F(S),
\]
and thus
\[
\theta(F_* {\cred (}x {f'_j}^p{\cred )})=u(F_*{\cred (}x{f'_j}^p\Delta_1(f^{p-1}){\cred )})=u(F_*{\cred (}x\Delta_1({f'_j}f^{p-1}){\cred )})-u(F_*{\cred (}xf^{p(p-1)}\Delta_1({f'}_j){\cred )}).
\]
Since $f'_jf^{p-1}={f'}_j^p(f/f'_j)^{p-1}$, we have
\[
u(F_*{\cred (}x\Delta_1(f'_jf^{p-1}){\cred )})=u(F_*{\cred (}x{f'_j}^{p^2}\Delta_1((f/f'_j)^{p-1}){\cred )}) \in ({f'_j}^p)\subseteq I^{[p]} \subseteq I_1.
\]
Moreover, we have $u(F_*{\cred (}xf^{p(p-1)}\Delta_1({f'}_j){\cred )})\in I_1$, and we conclude the assertion.
\end{claimproof}

Next, we take homogeneous elements $g_1,\ldots,g_{n-1}$ such that $u(F_*g_i)=0$, $\theta(F_*g_i) \equiv g_{i+1} \mod I_1$, $g_1 \in I_1$, and $\theta(F_*g_{n-1}) \notin \m^{[p]}$.
Since $\sht(S/I)=n$, we can take such a sequence.
We put
$h_{i}:=g_i-\theta(F_*g_{i-1})$
for $2 \leq i \leq n-1$.
\begin{claim}
The degree $d_i$ of $g_i$ is $(p-1)\mu$ for all $i$.
\end{claim}

\begin{claimproof}
First, we prove $d_{n-i} \leq (p-1)\mu$ for all $i$ by induction on $i$.
Since $\theta(F_*g_{n-1})$ is not contained in $\m^{[p]}$, we have
$d_{n-1} \leq (p-1)\mu.$

Next, we assume that $d_{n-i+1} \leq (p-1)\mu$ for some $i \geq 2$.
Then we have
\[
\deg(\theta(F_*g_{n-1+1})) \leq (p-1)\mu.
\]
Suppose that $d_{n-i} > (p-1)\mu$. Then $\deg(g_{n-i}) \neq \deg(\theta(F_*g_{n-i+1}))$.
Since $g_{n-i}=\theta(F_*g_{n-i+1})+h_{n-i}$, the element $g_{n-i}$ is the degree $d_{n-i}$-part of $h_{n-i}$.
In particular, $F_*g_{n-i}\in F_*I^{[p]} \cap v^\perp$.
By Claim \ref{claim:ip claim}, we have $\theta(F_*g_{n-i}) \in I_1$, and thus $\theta(F_*g_{n-1}) \in I_{i}$, a contradiction with $\sht(S/I)=n$. Therefore, we have $d_{n-i} \leq (p-1)\mu$.

By the above argument, we also obtain that $d_{i+1}=\deg(\theta(F_*g_i))$ and $h_i$ is a homogeneous of degree $d_i$.
Next, we prove $d_i=(p-1)\mu$ by induction on $i$.
When $i=1$, we obtain the assertion since $g_1\in I_1$.
Suppose $\deg(g_1) < (p-1)\mu$. Then $F_*g_1\in F_*I^{[p]} \cap v^{\perp}$, and Claim \ref{claim:ip claim} shows that $\theta(F_*g_1)\in I_1$.
Therefore, we have $g_2 \in I_1$.
Repeating this procedure, we have $\theta(F_*g_{n-2})\in I_{n-1}$, a contradiction with $\sht(S/I)=n$.
Therefore, $d_1=(p-1)\mu$ holds.
Next, if $d_{i-1}=(p-1)\mu$ for $i \geq 2$, 
then we have
\[
d_{i}=\deg(\theta(F_*g_{i-1}))=((p^2-1)\mu-(p-1)\mu)/p=(p-1)\mu,
\]
as desired.
\end{claimproof}

\begin{claim}
There exist $g'_1,\ldots,g'_{n-1}$ and $a_1,\ldots,a_{n-1} \in k$ such that $u(F_*g'_i)=0$, $g'_1=a_1f^{p-1}$, $g'_{i+1}=\theta(F_*g'_i)+a_if^{p-1}$ and $g_i \equiv g'_i \mod I^{[p]}$.
\end{claim}

\begin{claimproof}
We construct such elements inductively.
For $i=1$, since the degree of $g_1$ is $(p-1)\mu$, we have
$g_1=f^{p-1}a_1+q_1$,
where $a_1 \in k$ and $q_1 \in I^{[p]}$.
We put $g'_1:=a_1f^{p-1}$.
Since $g'_1 \notin \m^{[p]}$, we have $u(F_*g'_1)=0$ by Lemma \ref{lem:homogeneous of top degree}.
Therefore, $g'_1$ satisfies the desired result.
Next, we assume that we construct $g'_1,\ldots,g'_{i-1}$ and $a_1,\ldots,a_{i-1}$ for some $i \geq 2$.
We define $q_{i-1}$ by
$g_{i-1}=g'_{i-1}+q_{i-1}.$
Then $q_{i-1}$ is contained in $I^{[p]}$.
Furthermore, since 
$u(F_*g_{i-1})=u(F_*g'_{i-1})=0$,
we have $F_*q_{i-1} \in F_*I^{[p]} \cap v^\perp$.
By Claim \ref{claim:ip claim}, $\theta(F_*q_{i-1})$ is contained in $I_1$.
Since the degrees of $g_i$ and $\theta(F_*g'_{i-1})$ are $(p-1)\mu$ and we have
\[
g_i \equiv \theta(F_*g_{i-1}) \equiv \theta(F_*g'_{i-1}) \mod I_1,
\]
there exist $a_i \in k$ and $q_i \in I^{[p]}$ such that
\[
g_i=\theta(F_*g'_{i-1})+a_if^{p-1}+q_i.
\]
We put $g'_i:=\theta(F_*g'_{i-1})+a_if^{p-1}$, then $g_i \equiv g'_i \mod I^{[p]}$.
In particular, we have 
\[
u(F_*g'_i)=u(F_*g'_i)-u(F_*g_i) \in u(F_*I^{[p]}) \subseteq I \subseteq \m.
\]
Since the degree of $g'_i$ is $(p-1)\mu$, we have $u(F_*g'_i)=0$ by Lemma \ref{lem:homogeneous of top degree}.
Therefore, $g'_i$ satisfies the desired conditions.
\end{claimproof}

Replacing $g_i$ by $g'_i$, we may assume 
\[
\theta(F_*g_{n-1})=\theta^{n-1}(F^{n-1}_{\cred *} {\cred (}a_1f^{p-1}{\cred )})+\cdots+\theta(F_*{\cred (} a_{n-1}f^{p-1} {\cred )}) \equiv \theta^{n-1}(F^{n-1}_* {\cred (} a_1f^{p-1} {\cred )}) \mod \m^{[p]}.
\]
by Claim \ref{claim:inclusion}.
Since the left-hand side is not contained in $\m^{[p]}$, we have
\[
\theta^{n-1}(F_*{\cred (}f^{p-1}k {\cred )}) \nsubseteq \m^{[p]},
\]
as desired.
\end{proof}

\begin{thm}\label{thm:base change, Calabi-Yau case}
Let $k \subseteq K$ be a field extension and $S_K:=K[x_1,\ldots,x_N]$.
Let $f'_1,\ldots,f'_m$ be a homogeneous regular sequence in $S$, $I:=(f'_1,\ldots,f'_m)$, and $f:=f'_1 \cdots f'_m$.
Suppose that $f$ is a homogeneous element of degree $\mu$.
Then we have $\sht(S/I) =\sht(S_K/IS_K)$.
\end{thm}

\begin{proof}
By considering field extensions $k \subseteq K^{1/p^{\infty}}$ and $K \subseteq K^{1/p^{\infty}}$, we may assume that $K$ is perfect.
As in Convention \ref{conv:graded ring}, we define
\[
\Delta_{1,K} \colon S_K \to S_K/F(S_K)
\]
by using the basis
\[
\{F_* {\cred (} x_1^{i_1}\cdots x_N^{i_N} {\cred )} \mid 0 \leq i_j \leq p-1 \}.
\]
Then $y^{j_1}_1\cdots y^{j_s}_s x_1^{i_1} \cdots x_N^{i_N}$ is a $p$-monomial in $S_K$ since $y_l$ has a $p$-th root in $K$.
In particular, $p$-monomials in $S$ are also $p$-monomials in $S_K$.
Therefore, a $p$-monomial decomposition of every $a \in S$ gives a $p$-monomial decomposition in $S_K$, and thus we have
\[
\Delta_1(a) \equiv \Delta_{1,K}(a) \mod F(S_K).
\]
Therefore, if we take a representative of $\Delta_1(f^{p-1})$ which is homogeneous of degree $p(p-1)\mu$. 
Then its image in $S_{K}$ is a representative of $\Delta_{1,K}(f^{p-1})$, which is homogeneous of degree $p(p-1)\mu$.
Therefore, if we define the elements $f_n \in S$ as in Theorem \ref{thm:Fedder's criterion for Calabi-Yau},
then we have
\[
\sht(S/I)=\mathrm{inf}\{n \mid f_n \notin \m^{[p^n]} \}=\mathrm{inf}\{n \mid f_n \notin (\m S_K)^{[p^n]} \}= \sht(S_K/IS_K),
\]
as desired.
\end{proof}

\begin{lem}\label{lem:special inversion of adjunction}
Let $g_1,\ldots ,g_m$ be homogeneous elements of $S$ such that $f:=g_1 \cdots g_m$ is of degree $\mu$.
Assume that $\Delta_1(g_2)=\cdots=\Delta_1(g_m)=0$.
Then $\sht(S/f) \geq \sht(S/g_1)$.
In particular, if we further assume that $g_1,\ldots ,g_r$ is a regular sequence, then we have 
\[
\sht(S/(g_1,\ldots ,g_r)) \geq \sht(S/g_1).
\]
\end{lem}

\begin{proof}
Suppose $\sht(S/f)=1$. Then $f^{p-1} \notin \m^{[p]}$, and $g_1^{p-1} \notin \m^{[p]}$. Thus, $\sht(S/g_1)=1$.
We now assume that $\sht(S/f) \geq 2$.
By Proposition \ref{prop:rule for delta map} (4) and Remark \ref{rmk:same properties}
\[
\Delta_1(f^{p-1}) \equiv \sum_{1 \leq i \leq r} (f/g_i)^{p(p-1)}\Delta_1(g_i^{p-1}) = (f/g_1)^{p(p-1)}\Delta_1(g_1^{p-1}) \mod F(S).
\]
We write $\theta:=u(F_*\Delta_1(f^{p-1})\cdot \underline{\ \ })$ and $\theta_1:=u(F_*\Delta_1(g_1^{p-1})\cdot \underline{\ \ })$.
Then we have
\[
\theta=(f/g_1)^{p-1}\theta_1\ \text{on}\ {\cred v^\perp.}
\]
We define a sequence $\{g_{1,j}\}_j$ by $g_{1,1}:=f^{p-1}$ and
\[
g_{1,j}:=(f/g_1)^{p-1}\theta_1(F_*g_{1,j-1})
\]
We denote the minimum positive integer $j$ satisfying $u(F_*g_{1,j}) \neq 0$ by $h$.
Then, by the construction of $g_{1,j}$, we have $g_{1,h} \in I_{h}$, where $I_h$ is the ideal defined by taking $I=(g_1)$ in Corollary \ref{cor:graded Fedder}. 
Since $\deg(f)=\mu$, we have $\deg(\theta^{h-1}(F^{h-1}_*f^{p-1}))=(p-1)\mu$.
Moreover, we have
\[
g_{1,j}=\theta^{j-1}(F_*^{j-1}f^{p-1})
\]
for $j \leq h$.
Indeed, if the equation holds for $j-1$, then
\[
g_{1,j}= \theta_1(F_*g_{1,j-1})(f/g_1)^{p-1}=\theta(F_*g_{1,j-1})=\theta^{j-1}(F^{j-1}_*f^{p-1})
\]
where the second equation follows from $u(F_*g_{1,j-1})=0$ and the third equation follows from the induction hypothesis.
Therefore, the degree of $g_{1,h}$ is $(p-1)\mu$.
Since $u(F_*g_{1,h}) \neq 0$, it follows from Lemma \ref{lem:homogeneous of top degree} that $g_{1,h}\notin \m^{[p]}$.
Thus, $\sht(S/g_1)$ is at most $h$ by Theorem \ref{thm:Fedder's criterion}. 
On the other hand, since $h$ is the minimum integer with
\[
u(F_*{\cred (}\theta^{h-1}(F^{h-1}_*f^{p-1}){\cred )}) \notin \m,
\]
we conclude from Theorem \ref{thm:Fedder's criterion for Calabi-Yau} that $\sht(S/f)=h$.
\end{proof}

\begin{rmk}
Theorem \ref{thm:base change, Calabi-Yau case} does not hold in non-Calabi-Yau cases in general.
Indeed, some quasi-$F$-split wild conic bundle gives an example (see \cite{KTY2}).
\end{rmk}

In \cite{KTY2}, we prove a generalization of the following theorem to hypersurfaces in a multiprojective space or a weighted projective space.

\begin{thm}\label{thm:Fedder's criterion for projective varieties in weighted case}
Let $P:=\P^{N-1}_{[x_1:\cdots :x_N]}$ be a projective space over an $F$-finite field $k$ of characteristic $p>0$.
Let $S:=k[x_1,\ldots,x_N]$ be a section ring of $P$.
Let $X$ be a complete intersection in $P$ defined by a homogeneous regular sequence $f_1,\ldots,f_m \in S$ with $\mathrm{dim}(X) \geq 1$.
Then the following hold.
\begin{itemize}
    \item[\textup{(1)}] We define ideals $\{I_n\}$ as in Corollary \ref{cor:graded Fedder} by $f_1,\ldots,f_m$. Then we have
    \[
    \sht(X)=\mathrm{inf}\{n \mid I_n \nsubseteq \m^{[p]}\}.
    \]
    \item[\textup{(2)}] We assume the degree of $f_1\cdots f_m$ coincides with $N$ and define elements $\{f_n\}$ as in Theorem \ref{thm:Fedder's criterion for Calabi-Yau} by $f_1,\ldots,f_m$. Then we have
    \[
    \sht(X)=\mathrm{inf}\{n \mid f_n \notin \m^{[p^n]} \}
    \]
\end{itemize}
\end{thm}

\begin{proof}
By Corollary \ref{cor:graded Fedder} and Theorem \ref{thm:Fedder's criterion for Calabi-Yau}, it is enough to show that 
\[
\sht(X)=\sht(S/({\cred f_1},\ldots,f_r)).
\]
It follows from \cite[Corollary 7.17]{KTTWYY2} and Proposition \ref{prop:multigraded ring versus local ring}.
\end{proof}

\section{Examples}

\subsection{Calabi-Yau varieties}

In this subsection, we apply Theorem \ref{thm:Fedder's criterion for Calabi-Yau} to compute quasi-$F$-split heights of Calabi-Yau hypersurfaces.
{\cred All computations can be verified using the Macaulay 2 (based on Theorem \ref{thm:Fedder's criterion for projective varieties in weighted case}) code available on the second author’s homepage (\cite{Takamatsu_code}).}

\begin{eg}[Heights of RDP K3 surfaces]
In \cite[{\cred Section 4}]{Goto} and \cite{Yui}, Artin-Mazur heights of K3 surfaces obtained as the minimal resolution of some hypersurfaces in weighted projective 3-space are computed.
By Theorem \ref{thm:Fedder's criterion for projective varieties in weighted case}, we can compute quasi-$F$-split heights of some of their hypersurfaces.
For example, we can show that the height of the weighted Delsarte surface defined by
\[
x_{0}^{8} x_{1} + x_{1}^{6}x_{2} + x_{2}^{3} + x_{3}^{2}x_{0} = 0,
\]
is $8$ when $p=3$.
By \cite[Proposition {\cred 3.46} and Corollary {\cred 3.47}]{KTTWYY},
{\cred the height of the minimal resolution of this surface is also $8$.
This recovers the computation in \cite[Example 4.2]{Goto}.}
\end{eg}

\begin{eg}[Heights of K3 surfaces over $\mathbb{F}_3$]\label{example:char 3 K3 surface}
\,
In Table \ref{table:K3 surfacee}, we present examples of K3 surfaces (or more precisely, quartic K3 surfaces) over $\mathbb{F}_{3}$ of {\cred any given} height.
Note that such examples over $\mathbb{F}_{2}$ are already given in \cite{Kedlaya} by computing rational points of K3 surfaces.
\begin{table}[ht]
\caption{Artin-Mazur heights of K3 surfaces over $\mathbb{F}_3$}
\centering
\begin{tabular}{|c|l|}
\hline
ht &  \hspace{45mm} equation\\
\hline
\hline
$1$ & $x^4+y^4+z^4+2w^4+x^2yw+yz^2w$ \\
\hline
$2$ & $x^4+2y^4+2z^4+2w^4+xyz^2$ \\
\hline
$3$ & $x^4+y^4+z^4+w^4+x^2z^2+xyz^2+z^3w$ \\
\hline
$4$ & $x^4+y^4+z^4+w^4+x^2z^2+xyz^2$ \\
\hline
$5$ & $x^4+y^4+z^4+w^4+x^3z+z^3w+yz^2w+yzw^2$ \\
\hline
$6$ & $x^4+y^4+z^4+w^4+x^2z^2+x^2yz\ {\cred  + \, xz^3}$ \\
\hline
$7$ & $x^4+y^4+z^4+w^4+xy^2z+xz^2w+yzw^2+y^2zw$ \\
\hline
$8$ &  $x^4 + x^2yz + x^2yw + 2x^2z^2 + xyw^2 + 2y^4 + y^3w + z^4 + w^4$\\
\hline
$9$ & $x^4+y^4+z^4+w^4+xy^3+y^3w+z^2w^2+2xyz^2+yzw^2$ \\
\hline
$10$ &  
\begin{tabular}{l}
$    x^4+2x^2yz+x^2yw+xy^2w+y^4+y^3w+y^2z^2$ \\
$+2y^2zw+y^2w^2+yz^3+ yz^2w+yzw^2+z^4+zw^3 $ \\
\end{tabular}\\
\hline
$\infty$ & $x^4+y^4+z^4+w^4$ \\
\hline
\end{tabular}
\label{table:K3 surfacee}
\end{table}

\end{eg}

\begin{eg}[Calabi-Yau threefolds of large height]\label{eg:ht=60}
The Calabi-Yau threefold (or more precisely, the smooth quintic threefold) over $\mathbb{F}_{2}$ defined by 
\[
\{x^5+y^5+z^5+w^5+u^5+xz^3w+yzw^3+x^2zu^2+y^2z^2w+xy^2wu+yzwu^2=0\}
\]
has the Artin-Mazur height $60$. 
This is the largest Artin-Mazur height example of all the smooth quintic threefolds we have found. In general, the height of a smooth Calabi-Yau threefold is known to be less than or equal to $102$, and we do not know if there exists one whose height is bigger than $60$.
\end{eg}

\subsection{Calabi-Yau hypersurfaces of arbitrary high Artin-Mazur heights over $\F_2$}
In this subsection, for given $h$, we obtain explicit defining equations of Calabi-Yau hypersurfaces over $\F_2$ of quasi-$F$-split heights $2h$.
In \cite{KTY2}, we construct a Calabi-Yau hypersurface over $\var{\F}_p$ whose Artin-Mazur height $h$ for every positive integer $h\in\Z_{>0}$ and {\cred every} prime number $p$. 
\begin{defn}\label{defn:height for elements}
Let $k$ be a perfect field {\cred of characteristic $p>0$}.
We put $S:=k[x_1,\ldots,x_N]$.
We define the basis of $F_*S$ over $S$ and the map $\Delta_1$ by using variables as in Convention \ref{conv:graded ring}.
Furthermore, for $f \in S$, we put
\[
\theta_f \colon F_*S \to S\ ;\ a \mapsto u(F_*{\cred (}\Delta_1(f^{p-1})a{\cred )}),
\]
where $u$ is defined in Convention \ref{conv:graded ring}.
We take a homogeneous element $f \in S$ of degree $N$ and an element $a \in S$.
We define the \emph{height of $a$ with respect to $f$} by
\[
\sht_f(a):=\mathrm{inf}\{h \mid \theta_f^{h-1}(F^{h-1}_*a) \notin \m^{[p]} \}.
\]
\end{defn}

\begin{rmk}\label{rmk:example}
By Theorem \ref{thm:Fedder's criterion for Calabi-Yau}, we have
$\sht(S/f)=\sht_f(f^{p-1})$.
Furthermore, by Lemma \ref{lem:homogeneous of top degree}, for a homogeneous element $a \in S$ of degree $(p-1)N$,
\[
\sht_f(a)=\mathrm{inf}\{h \mid u(F_*{\cred (}\theta_f^{h-1}(F^{h-1}_*a){\cred )}) \neq 0 \}.
\]
\end{rmk}

\begin{lem}\label{lem:height criterion for element}
We use the notation of Definition \ref{defn:height for elements}.
Let $\alpha$ be a homogeneous element of degree $N$ and 
\[
\alpha=M_1+\cdots+M_n
\]
the monomial decomposition.
Suppose that $k=\F_2$.
For a positive integer $h$, if
\[
\mathrm{min}\{i \mid \sht_f(M_i)\}=h\ \text{and}\  \# \{i \mid \sht_f(M_i)\}\ \text{is odd},
\]
then $\sht_f(\alpha)=h$.
\end{lem}

\begin{proof}
Since the minimum of the heights of $M_i$ is $h$, it follows that $\theta^{h-1}(F^{h-1}_*M_i) \neq v$ holds if and only if $\sht_f(M_i)=h$ for all $i$ as $k=\F_2$.
Since the number of such $i$ is odd, we have $\theta^{h-1}(F^{h-1}_*\alpha)=v$, and thus $\sht_f(\alpha)$ is at most $h$.
Since $u(F_* {\cred (}\theta_f^{h'-1}(F^{h'-1}_*M_i) {\cred )})=0$ for all $i$ and $h'<h$, we have $u(F_* {\cred (}\theta_f^{h'-1}(F^{h'-1}_*\alpha){\cred )})=0$, thus $\sht_f(\alpha)$ is $h$.
\end{proof}

\begin{eg}[Unboundedness of height for Calabi-Yau varieties in $p=2$]\label{eg:unboundedness}
Let $h$ be a positive integer.
We put $N:=2^h+1$ and $S=\F_2[a,b,c,x_1,\ldots,x_{N-3}]$.
Furthermore, we define a sequence of integers $N_0,\ldots,N_{h-1}$ satisfying $N_0=N-2$ and $N_{i-1}-N_{i}=2^i$ for $1 \leq i \leq h-1$.
We note that
\[
2+2^2+\cdots+2^{h-2}+2^{h-1}=2^h-2=N-3,
\]
thus $N_{h-1}=1$, $N_{h-2}=3$, $N_{h-3}=7$, for example.
We define a homogeneous element $f$ of degree $N$ by
\[
f:=a^N+b^N+c^N+x_1^N+\cdots+x_{N-3}^N+(b+c)g,
\]
where $g=0$ if $h=1$ and 
\[
g=c^2x_1\cdots x_{N-3}+x_{N_1}^2 \cdots x_{N-3}^2+x_{N_2}^4 \cdots x_{N_1-1}^4+ \cdots + x_1^{2^{h-1}} x_2^{2^{h-1}}
\]
{\cred if $h\neq 1$.}
We note that $g$ is homogeneous of degree $2^h=N-1$.
We prove that $\Proj(S/f)$ is a smooth Calabi-Yau variety of height $2h$.
\begin{claim}
$\Proj(S/f)$ is smooth.
\end{claim}
\begin{claimproof}
Since
$\frac{\partial (bg)}{\partial b}=g$ and $\frac{\partial (cg)}{\partial c}=g$,
we have
\[
\frac{\partial f}{\partial b}=b^{2^h}+g,\ \ \frac{\partial f}{\partial c}=c^{2^h}+g.
\]
Therefore, at {\cred a} singular point, we have $b=c$.
Since 
\[
\frac{\partial f}{\partial x_i}=x_i^{2^h}+(b+c) {\cred \frac{\partial g}{\partial x_i}},
\]
at {\cred a} singular point, $x_i=0$ for all $i$, and in particular, $g=0$.
Therefore, we have $b=c=0$.
Since we have $\frac{\partial f}{\partial a}=a^{2^h}$,
we obtain that the singular points of {\cred $\Spec S/f$} is the origin, thus $\Proj(S/f)$ is smooth.
\end{claimproof}\\
By the claim, $\Proj(S/f)$ is a Calabi-Yau variety.
Next, we compute the quasi-$F$-split height of $S/f$.
If $h=1$, then $f=a^3+b^3+b^3$, thus the assertion follows from Example \ref{eg:easy example}.
Therefore, we assume that $h$ is at least $2$.
\begin{claim}\label{claim:computation of heights of elements for r is at most h}
Let $\alpha \in S$ be a monomial and $1 \leq r \leq h$ an integer.
Then $\sht_f(\alpha)=r$ if and only if 
\[
\alpha=a^{2^r-1}bcx_1^{2^{r-1}} \cdots  x^{2^{r-1}}_{N_{r-1}-1}.
\]
\end{claim}
\begin{claimproof}
We put
\[
\alpha_r:=a^{2^r-1}bcx_1^{2^{r-1}} \cdots  x^{2^{r-1}}_{N_{r-1}-1}.
\]
We prove the statement by induction on $r$.
For $r=1$, it follows from 
\[
\alpha_1=abcx_1\cdots x_{N-3}.
\]
For $r \geq 2$, by the induction hypothesis, $\sht_f(\alpha)=r$ if and only if $\theta_f(F_*\alpha)$ contains the term $\alpha_{r-1}$.
Therefore, it is equivalent to the condition that $\Delta_1(f)\alpha$ contains
\[
\beta_r:=\alpha_{r-1}^2v=a^{2^r-1}b^3c^3x_1^{2^{r-1}+1} \cdots x_{N_{r-2}-1}^{2^{r-1}+1} x_{N_{r-2}} \cdots x_{n-3}.
\]
We put
\[
f':=a^N+b^N+c^N+x_1^N+\cdots+x_{n-3}^N.
\]
Then
\[
\Delta_1(f)=\Delta_1(f')+(b^2+c^2)\Delta_1(g)+bcg^2+(b+c)gf'.
\]
Since $r \leq h$, we have $2^r-1 \leq 2^h-1 < N$.
In the terms appearing $\Delta_1(f)$,
\[
b^2 \cdot c^2x_1 \cdots x_{n-3} \cdot bx_{N_{r-1}}^{2^{r-1}} \cdots x_{N_{r-2}-1}^{2^{r-1}} 
\]
only contributes to the equivalent condition.
Thus, the height of $\alpha$ is $r$ if and only if $\alpha=\alpha_r$.
\end{claimproof}\\
We note that $\alpha_h=a^{2^h-1}bc$.
Next, we consider the $h=2$ case.
In this case,
\[
f=a^5+b^5+c^5+x_1^5+x_2^5+(b+c)(c^2x_1x_2+bx_1^2x_2^2).
\]
By the above claim, it is enough to show that $\theta_f^2(F^2_*f)$ contains 
\[
\alpha_2=a^3bc.
\]
For a monomial $\alpha \in S$, the height of $\alpha$ is three if and only if $\Delta_1(f)\alpha$ contains $a^7b^3c^3x_1x_2$.
Therefore, we have
\[
\alpha=a^2b^3,\ \mathrm{or}\ a^2b^2c.
\]
Since $f\Delta_1(f)$ contains $a^5 \cdot b^5 \cdot c^3x_1x_2$ and the heights of the other terms are not four, the height of $f$ is four.

From now on, we assume that $h$ is at least three.
\begin{claim}
Let $\alpha \in S$ be a monomial and $h+1 \leq r \leq 2h-1$ an integer.
We put $\rho_r:=2^{r}-2^{r-1}-\cdots-2^{h+1}-2^{h}-2^{r-h}$ and
\[
\alpha_r:=a^{\rho_r}b^{2^{r-h}}c,\ \ \alpha'_r:=a^{\rho_r}b^{2^{r-h}+1}.
\]
Then the height of $\alpha$ is $r$ if and only if $\alpha=\alpha_r$ or $\alpha'_r$.
\end{claim}
\begin{claimproof}
We prove the claim by induction on $r$.
For $r=h+1$,
by Claim \ref{claim:computation of heights of elements for r is at most h}, the height of $\alpha$ is $h+1$ if and only if $\Delta_1(f)\alpha$ contains
\[
\beta_{h+1}:=a^{2^{h+1}-1}b^{3}c^3x_1\cdots x_{N-3}.
\]
Therefore, we see that 
\[
\alpha=a^{2^{h+1}-2^h-2}b^2c,\ \ \mathrm{or}\ a^{2^{h+1}-2^h-2}b^3.
\]
Since $\rho_{h+1}=2^{h+1}-2^{h}-2$, we have $\alpha=\alpha_{h+1}$ or $\alpha'_{r+1}$.
For $r \geq h+2$, by the induction hypothesis, the height of $\alpha$ is $r$ if and only if $\Delta_1(f)\alpha$ contains 
\[
\beta_r:=\alpha_r^2v\ \ \mathrm{or}\ \beta'_r:=(\alpha'_r)^2v.
\]
If $\Delta_1(f)\alpha$ contains $\beta_r$, then
\[
\alpha=a^{2\rho_{r-1}-2^h}b^{2^{r-h}}c\ \ \mathrm{or}\ a^{2\rho_{r-1}-2^h}b^{2^{r-h}+1}.
\]
Since $2\rho_{r-1}-2^h=\rho_{r}$, it follows that $\alpha$ is either $\alpha_r$ or $\alpha'_r$.
Next, we consider the case where $\Delta_1(f)\alpha$ contains $\beta'_r$.
In this case, we have $2^{r-h}+3 \geq n$.
It means that $h=2$ and $r=2h-1=3$ by using $r \leq 2h-1$.
By assumption, this case does not occur.
\end{claimproof}\\
By the above claim, for a monomial $\alpha \in S$, the height of $\alpha$ is $2h$ if and only if $\Delta_1(f) \alpha$ contains
\[
\beta_{2h}:=\alpha_{2h-1}^2v\ \ \mathrm{or}\ \beta'_{2h}:=(\alpha'_{2h-1})^2v.
\]
Since 
\[
2\rho_{2h-1}+1=2^{2h}-2^{2h-1}-\cdots-2^{h+1}-2^h+1=2^h+1=N,
\]
we have
\[
\beta_{2h}=a^Nb^Nc^3x_1\cdots x_{N-3},\ \ \beta'_{2h}=a^nb^{N+2}cx_1 \cdots x_{N-3}.
\]
Thus, if $\alpha$ appears in $f$, then $\alpha$ is either $a^N$, $b^N$ or $c^3x_1 \cdots x_{N-3}$.
Therefore, $f$ contains only three terms whose height is $2h$.
Moreover, by the above claims, $f$ does not contain the term whose height is less than $2h$.
By Lemma \ref{lem:height criterion for element}, the quasi-$F$-split height of $f$ is $2h$, that is, the height of $S/f$ is $2h$.

\end{eg}

\printbibliography
\end{document}